\documentclass{amsart}
\usepackage{latexsym}
\usepackage{amsmath}
\usepackage{amssymb}
\usepackage[all]{xy}
\usepackage{autobreak}
\usepackage{enumitem}
\usepackage[hypertexnames=false]{hyperref} 
\usepackage[capitalize]{cleveref}
\crefname{equation}{}{}
\crefname{enumi}{}{}
\setlist[enumerate]{label=(\arabic{enumi})}

\usepackage[\ifdefined\enabletodonotes\else disable\fi]{todonotes}

\usepackage{autonum} 

\usepackage{tikz}
\usetikzlibrary{positioning}
\usetikzlibrary{arrows.meta}

\numberwithin{equation}{section}
\newtheorem{thm}{Theorem}
\numberwithin{thm}{section}
\newtheorem{prop}[thm]{Proposition}

\newtheorem{cor}[thm]{Corollary}
\newtheorem{lem}[thm]{Lemma}
\theoremstyle{definition}
\newtheorem{defn}[thm]{Definition}

\theoremstyle{remark}
\newtheorem{rem}[thm]{Remark}
\newtheorem{ex}[thm]{Example}
\newtheorem{prob}[thm]{Problem}
\newcommand{\aut}[1]{\text{\rm aut}_1({#1})}
\newcommand{\K}{{\mathbb K}}
\newcommand{\C}{{\mathbb C}}
\newcommand{\Q}{{\mathbb Q}}
\renewcommand{\ker}{\operatorname{Ker}}
\newcommand{\im}{\operatorname{Im}}

\newcommand{\SpaceModel}{\wedge V}
\newcommand{\HochschildModel}{\mathcal L}

\newcommand{\cyclicModel}{\mathcal E}

\newcommand{\ev}{{\rm ev}}
\newcommand{\ad}{{\rm ad}}
\newcommand{\pr}{{\rm pr}}
\newcommand{\inc}{{\rm inc}}
\newcommand{\Der}{\operatorname{Der}}
\newcommand{\Hom}{\operatorname{Hom}}
\newcommand{\End}{\operatorname{End}}
\newcommand{\Map}{\operatorname{Map}}
\newcommand{\pd}{\operatorname{pd}}
\newcommand{\PD}{\operatorname{PD}}

\newcommand{\cyclic}[1]{CC(#1)}
\newcommand{\dcyc}{d_u}
\newcommand{\cycl}{\overline{L}} 
\renewcommand{\deg}[1]{\mathopen\vert #1\mathclose\vert}

\newcommand{\mixed}{{\mathcal C}}
\newcommand{\done}{d}
\newcommand{\dtwo}{B}

\newcommand{\G}{\mathfrak g}
\newcommand{\dG}{\delta}
\renewcommand{\H}{\mathfrak h}

\newcommand{\lie}{\mathfrak{g}}
\newcommand{\dlie}{\delta}
\newcommand{\ccl}{L} 
\newcommand{\cce}{e} 
\newcommand{\ccs}{S} 
\newcommand{\cct}{T} 

\newcommand{\mapright}[1]{%
  \smash{\mathop{%
      \hbox to 1cm{\rightarrowfill}}\limits_{#1}}}
\newcommand{\maprightd}[2]{%
  \smash{\mathop{%
      \hbox to 1.2cm{\rightarrowfill}}\limits^{#1}\limits_{#2}}}
\newcommand{\mapleft}[1]{%
  \smash{\mathop{%
      \hbox to 1cm{\leftarrowfill}}\limits_{#1}}}
\newcommand{\mapleftu}[1]{%
  \smash{\mathop{%
      \hbox to 0.8cm{\leftarrowfill}}\limits^{#1}}}
\newcommand{\maprightu}[1]{%
  \smash{\mathop{%
      \hbox to 1cm{\rightarrowfill}}\limits^{#1}}}
\newcommand{\maprightud}[2]{%
  \smash{\mathop{%
      \hbox to 1cm{\rightarrowfill}}\limits^{#1}_{#2}}}
\newcommand{\mapleftud}[2]{%
  \smash{\mathop{%
      \hbox to 1cm{\leftarrowfill}}\limits^{#1}_{#2}}}

\title[Cartan calculi on the free loop spaces]{Cartan calculi on the free loop spaces \\
}

\author[K. Kuribayashi]{Katsuhiko Kuribayashi}
\address{%
  Department of Mathematical Sciences,
  Faculty of Science,
  Shinshu University,
  Matsumoto, Nagano 390-8621, Japan
}
\email{kuri@math.shinshu-u.ac.jp}

\author[T. Naito]{Takahito Naito}
\address{%
  Nippon Institute of Technology,
  Gakuendai, Miyashiro-machi, Minamisaitama-gun, Saitama 345-8501, Japan
}
\email{naito.takahito@nit.ac.jp}

\author[S. Wakatsuki]{Shun Wakatsuki}
\address{%
  Graduate School of Mathematics, Nagoya University,
  Furo-cho, Chikusa-ku, Nagoya, Aichi 464-8601, Japan
}
\email{shun.wakatsuki@math.nagoya-u.ac.jp}

\author[T. Yamaguchi]{Toshihiro Yamaguchi}
\address{%
  Faculty of Education,
  Kochi University, Akebono-cho, Kochi 780-8520, Japan
}
\email{tyamag@kochi-u.ac.jp}

\thanks{%
  The first author was partially supported by a Grant-in-Aid for Scientific
  Research (B) 21H00982 from Japan Society for the Promotion of Science.
  The second author was supported by JSPS KAKENHI Grant Number JP18K13403.
  The third author was supported by JSPS KAKENHI Grant Number 20J00404.
  The fourth author was partially supported by JSPS KAKENHI Grant Number 20K03591.
}

\begin{document}

\maketitle

\begin{abstract} A typical example of a Cartan calculus consists of the Lie derivative and the contraction with vector fields of a manifold on the derivation ring of the de Rham complex.
  In this manuscript, a {\it second stage} of the Cartan calculus is investigated.
  In a general setting, the stage is formulated with operators obtained by  the Andr\'e--Quillen cohomology of a commutative differential graded algebra $A$
on the Hochschild homology of $A$ in terms of the homotopy Cartan calculus in the sense of Fiorenza and Kowalzig.
  Moreover, the  Cartan calculus is interpreted geometrically
  with maps from the rational homotopy group of the monoid of self-homotopy equivalences
  on a space $M$ to the derivation ring on the loop cohomology of $M$. We also give a geometric description to Sullivan's isomorphism, which relates the geometric Cartan calculus to the algebraic one, via the $\Gamma_1$ map due to F\'elix and Thomas.
\end{abstract}

\tableofcontents

\section{Introduction}


In the previous work \cite{KNWY}, we consider a method to describe the string bracket \cite{CS} on the rational $S^1$-equivariant homology of the free loop space $LM$ of
a simply-connected manifold $M$ in terms of the Gerstenhaber bracket on the loop homology of $M$, namely, the homology of $LM$. In particular, the reduction is possible if $M$ is {\it BV exact}; that is, the reduced Batalin--Vilkovisky (BV) operator on the loop homology is exact;
see \cite[Definition 2.9, Theorem 2.15 and Corollary 2.16]{KNWY} for more details.

The result \cite[Assertion 1.2]{KNWY} summarizes relationships between the BV exactness and other traditional homotopy invariants containing the formality of a space.
Especially, we show that a simply-connected space is BV exact if the space admits {\it positive weights}; see \cite[Theorem 2.21]{KNWY}. The key to proving the theorem is that two particular derivations on a Sullivan algebra associated with the space satisfy the {\it Cartan magic formula}; see \cref{prop:CartanCalculusSullivanModel} for the derivations that we use therein.  
The appearance of the formula has inspired us to consider algebraic and topological backgrounds for the derivations.
In this article,
we investigate such derivations in the framework of {\it homotopy Cartan calculi} introduced by Fiorenza and Kowalzig \cite{FK} and moreover give geometric descriptions to the \text{\it Lie derivative} and the {\it contraction operator}, which induce the two derivations mentioned above.

In order to describe our results in more detail, we first recall the classical Cartan calculus of the differential forms on a manifold $M$ together with
Connes' result on the Hochschild cohomology.
The space of vector fields on $M$ is considered a Lie algebra $\Der(C^\infty(M))$ of derivations on $C^\infty(M)$ the
$\mathbb{R}$-algebra of smooth functions on $M$.
The result \cite[II Section 6. Example]{Connes} due to Connes asserts that the {\it continuous} Hochschild cohomology $(HH^*_{conti}(C^\infty(M)), B)$ with Connes' $B$-operator $B$ is isomorphic to the de Rham complex $(\Omega^*(M), d)$ as a complex provided $M$ is compact. Thus the Lie  derivative $L_X$ and the contraction  (interior product) $\iota_X$ for each vector field $X$
are incorporated in the framework of a {\it Cartan calculus}
\begin{equation}\label{eq:classicalCatranCal}
  \xymatrix@C20pt@R20pt{
    \Der(C^\infty(M))\ar[r]^-{L_{(\ )}}_-{\iota_{( \ )}} &\big(\Der(\Omega^*(M)), d \big)\cong\big(\Der(HH^*_{conti}(C^\infty(M))), B\big)
  }
\end{equation}
in the sense that $L_{( \ )}$ is a Lie algebra representation and $\iota_{( \ )}$ is a linear map which satisfy, for any vector field $X$,
Cartan's magic formula
\[
L_X = [d, \iota_X].
\]

The Andr\'e--Quillen cohomology $H^{-*}_{AQ}(A)$ of a commutative differential graded algebra $A$ is an important invariant for such differential objects; see, for example, \cite{B-L05} for its applications.
Thus we may apply again cyclic theory, namely cyclic homology and Hochschild homology to the de Rham complex $(\Omega^*(M), d)$ involving the Andr\'e--Quillen cohomology. Let $\Der{(A)}$ denote the derivation subalgebra of the endomorphism Lie algebra $\End{(A)}$ of a  differential graded algebra $A$.
While the assignment $\Der{( \ )}$ is not functorial,
the Andr\'e--Quillen cohomology is defined as a {\it derived   version} of $\Der (A)$; see \cite{B-L05} and also \cref{sect:preliminaries} for the definition.

Let $M$ be a simply-connected manifold and
$\aut M$ the monoid of self-homotopy equivalences on $M$.
Then, we obtain the isomorphism
\[
\xymatrix@C20pt@R20pt{
\Phi : \pi_*(\aut M)\otimes {\mathbb R} \ar[r]^-{\cong} &
H^{-*}_{AQ}(\Omega^*(M))
}
\]
of Lie algebras due to Sullivan \cite{S}; see \cite[Theorems 3.6 and 4.3]{FLS} and
also \cref{sect:preliminaries} for the definition of $\Phi$.
Here, the homotopy group $\pi_*(\aut M)$ is regarded as a Lie algebra endowed with the Samelson product.
Main results (Propositions \ref{prop:LieAlg-New}, \ref{prop:Lie_e} and \cref{thm:e_and_L}) in this article
enable us to obtain a Cartan calculus on the de Rham complex $\Omega^*(M)$ with values in the endomorphism ring of the Hochschild homology of $\Omega^*(M)$
and its geometric interpretation with the free loop space $LM$. More precisely, the assertions are summarized as follows.

\begin{thm}\label{thm:CartanCal} Under the same notations and assumptions as above,
  there exists a commutative diagram
  \[
    \xymatrix@C25pt@R20pt{
      H^{-*}_{AQ}(\Omega^*(M)) \ar[r]^-{{}_aL}_-{e} & \left(\End(HH_*(\Omega^*(M))), B\right) \\
      \pi_*(\aut M)\otimes {\mathbb R} \ar[r]^-{L}_-{(-1)^*e} \ar[u]^-{\Phi}_-{\cong} & \left( \Der(H^*(LM; {\mathbb R})), \Delta \right)
      \ar[u]_-{\ell}
    }
  \]
  for $*> 1$ in which the upper row sequence is a Cartan calculus induced by a homotopy Cartan calculus in the sense of Fiorenza and Kowalzig \cite{FK} and the bottom row sequence is a Cartan calculus given geometrically by applying the loop construction to the adjoint of an element of homotopy group of $\aut M$;
see (\ref{defn:L_geom}) and (\ref{defn:e_geom}).
In particular, the calculi give the formulae
  \[
    {}_aL_\eta = [B, e_\eta] \  \ \text{and} \ \ L_\theta = [\Delta, \pm e_\theta]
  \]
  for $\eta \in H^*_{AQ}(\Omega^*(M))$ and  $\theta \in \pi_*(\aut M)\otimes {\mathbb R}$.
  Here $B$ and $\Delta$ denote Connes' B-operator on the Hochschild homology and
   the Batalin--Vilkovisky operator on the loop cohomology, respectively. Moreover, the right vertical map $\ell$ is a monomorphism 
   induced by the isomorphism between the loop cohomology and the Hochschild homology in \cite[Theorem 2.4]{BV88} preserving operators $\Delta$ and $B$.
\end{thm}

\begin{rem} The contraction operator $e$ in the upper sequence in \cref{thm:CartanCal} is defined for $*\geq 1$. However, the operator $e$ in the bottom sequence is defined for $* >1$; see (\ref{defn:e_geom}) below.

We observe that the square above for
${}_aL$ and $L$ is commutative even if $* =1$; see the proof of \cref{thm:CartanCal} in the end of \cref{sect:geom_L_e}. Moreover, it follows from \cref{lem:Lie_homo} and \cref{thm:LieAlgL} \cref{item:LieAlgL-i} that the maps ${}_aL$ and $L$ are morphisms of Lie algebras, respectively.
\end{rem}

We give more comments on Theorem \ref{thm:CartanCal} and its related results in this article.  By the Cartan calculus in (\ref{eq:classicalCatranCal}), we can regard
the de Rham complex as appearing via the Hochschild homology theory for $C^\infty(M)$.
The calculus in the upper sequence in the theorem is obtained by applying again the Hochschild homology theory to the de Rham complex.
Therefore, it seems that the pair
$({}_aL, e)$ of maps is in a {\it second stage} of Cartan calculi for the manifold $M$.

On the other hand, the dual of the calculus in the lower sequence in the theorem gives a Cartan calculus on the homology $H_*(LM;{\mathbb R})$.
The second author reveals a relationship between the calculus and algebraic structures in string topology theory; see \cite[Theorem 1.1]{Nai24}.

We stress that the pair in the upper sequence consists of the {\it Lie derivative} ${}_aL$ and the {\it contraction operator} $e$ in the homotopy Cartan calculus \cite[Definitions 3.1 and 3.7]{FK} associated with
the Hochschild complex ${\mathcal H}$ and the Burghelea--Vigu\'e-Poirrier complex $\HochschildModel$ of $\Omega^*(M)$ in \cite{BV88}, respectively.
As a consequence, we see that the two homotopy calculi coincide with each other on homology level; see Theorem \ref{thm:e_and_L}.
We remark that the Cartan magic formula holds in the complex $\HochschildModel$ before taking homology, but not in ${\mathcal H}$ in general;
see Propositions \ref{prop:CartanCalculusSullivanModel} and \ref{prop:HCartan_Hochschild}. 
Moreover, it is worth mentioning that
the contraction $e$ for the complex ${\mathcal H}$ is defined with the cap product between the Hochschild cochain and chain complexes of a commutative differential graded algebra; see, for example,
\cite{Luc2011} for the cap product.



Moreover, the contraction operator \(e\) is non-trivial in the following sense:
\begin{thm}\label{thm:e-inj}
  For a simply-connected closed manifold \(M\),
  the contraction operator
  \(e\colon \pi_*(\aut M)\otimes {\mathbb R} \to \Der(H^*(LM; {\mathbb R})) \)
  is injective.
\end{thm}
This is an immediate consequence of \cref{cor:e_faithful} and \cref{prop:Lie_e}.
The former is proved by showing that
the map invariably detects the fundamental class of a manifold \(M\);
see \cref{thm:e_fundamental_class}.

As for Sullivan's isomorphism $\Phi$, we show that the isomorphism factors through the map $\Gamma_1$ from $\pi_*(\aut M) \cong \pi_{*-1}(\Omega\aut M)$ to the loop homology of  simply-connected closed manifold $M$
introduced by F\'elix and Thomas in \cite{FT04}; see Theorem \ref{thm:Model2_aut}. Since the map $\Gamma_1$ is induced by
the evaluation map of the space of sections of the evaluation fibration $\ev_0 : LM \to M$,
it can be said that we give the isomorphism $\Phi$ a geometric interpretation.
It is worth mentioning that the Brown--Szczarba model \cite{BS97} for a function space plays a vital  role in the argument on the geometric description of the isomorphism $\Phi$.

We give comments on the Andr\'e--Quillen cohomology.
As mentioned above, taking the derivation algebra $\Der{(A)}$ for a commutative differential graded algebra $A$ is not functorial. However, we see that a Sullivan model
$\varphi\colon (\wedge V, d)\xrightarrow{\simeq} (A, d)$
 induces a morphism $\widetilde{\varphi} :
H^*(\Der(A)) \to H^*(\Der(\wedge V))$ which is compatible with Cartan calculi for $A$ and $\wedge V$. This is attained in \cref{prop:relate_Cartan_calculi}.
Such a map $\widetilde{\varphi}$ induced by $\varphi$ is an isomorphism if the codomain $A$ is also a Sullivan algebra; see \cref{cor:Der_quasi-iso}.
However, a quasi-isomorphism \(\varphi\colon (\wedge V, d)\xrightarrow{\simeq} (A, d)\) does not necessarily induce a quasi-isomorphism between \(\Der{(A)} \) and
\(\Der{(\wedge V)} \) in general; see \cref{rem:CP^2}.


The rest of this manuscript is organized as follows. 
Section \ref{sect:preliminaries} recalls results in rational homotopy theory with which we develop our arguments.
In Section \ref{sect:AlgCartanCal}, we recall the homotopy Cartan calculus mentioned above. Important examples of the calculi which come from a Sullivan algebra and the Hochschild complex of a DGA are given.
The naturality of a Cartan calculus are discussed in \cref{sect:naturality_of_CartanCalculus}.
Section \ref{sect:GeometricDiscription} is devoted to investigating geometric descriptions of the homotopy Cartan calculi considered in Section \ref{sect:Cartan}.  In Section \ref{sect:GeometricDiscription}, after explaining geometric constructions of the operations $L$ and $e$, we prove Theorem \ref{thm:CartanCal}.
In the rest of the section, we elaborate the proof of \cref{thm:Model2_aut} mentioned above.
Section \ref{sect:examples} deals with computational examples of the Lie derivative $L$ and the contraction operator $e$ described in Theorem \ref{thm:CartanCal}.

 In Appendix \ref{app:ModelAd}, we give a Sullivan representative for an adjoint map by using
{\it twice} Brown--Szczarba models for function spaces. The result plays a crucial rule in giving the geometric description of Sullivan's isomorphism $\Phi$.
In Appendix \ref{app:bar{L}}, we discuss an extension of the Lie derivative $L$ to cyclic theory and its geometric counterpart with the cobar-type Eilenberg--Moore spectral sequence converging to the $S^1$-equivariant cohomology of the free loop space $LM$; see Theorem \ref{thm:LieAlg-New} and Proposition \ref{prop:EMSS}.

\section{Preliminaries}\label{sect:preliminaries}
We begin with the definitions of the Hochschild complex of a differential non-negatively graded algebra (DGA for short) over a field, the endomorphism ring of a DGA and Sullivan's isomorphism $\Phi$ which are used repeatedly in this manuscript.
We assume that the underlying field is of characteristic zero unless otherwise stated.

Let $A=(A,d)$ be an augmented DGA, which is not necessarily graded commutative.
We use the cohomological grading on $A$ and then $\text{deg} \ d = +1$. While the homological degree of a graded vector space $W_*$ is also used, we freely apply the translation for homological and cohomological degrees with
$W_* = W^{-*}$.

Let $C_* (A) = (A\otimes T(s\bar{A}), d=d_1 + d_2)$ be the Hochschild chain complex of $(A,d)$.
Here, $\bar{A}$ denotes the augmentation ideal of $A$ and $s\bar{A}$ denotes the suspension of $\bar{A}$; that is, $(s\bar{A})^n = \bar{A}^{n+1}$.
The differentials $d_1$ and $d_2$ are defined by $d_1 = \sum_i d_{1,i}$ and $d_2 = \sum_i d_{2,i}$ with
\begin{align}
  &d_{1,i}(a_0 [a_1 | \cdots | a_n]) = \left\{
    \begin{array}{ll}
      da_0 [a_1 | \cdots | a_n] & (i=0),\\
      (-1)^{\varepsilon_i + 1}a_0 [a_1 | \cdots | da_i | \cdots | a_n] & (0<i \leq n),
    \end{array}
                                                      \right.
  \\
  &d_{2,i}(a_0 [a_1 | \cdots | a_n]) = \left\{
    \begin{array}{ll}
      (-1)^{|a_0|}a_0 a_1[a_2| \cdots | a_n] & (i=0),\\
      (-1)^{\varepsilon_{i+1}}a_0 [a_1 | \cdots | a_i a_{i+1} | \cdots | a_n] & (0<i < n),\\
      (-1)^{\varepsilon_n |sa_n| + 1} a_k a_0 [a_1 | \cdots | a_{n-1}] & (i=n),
    \end{array}
                                                           \right.
\end{align}
where $\varepsilon_i = |a_0| + \sum_{j<i}|sa_j|$.

\begin{defn}
  \label{defn:mixedObjects}
  \begin{enumerate}
    \item Let \((\mixed, \done)\)
      be a cochain complex.
      A triple \((\mixed, \done, \dtwo)\) is a \textit{mixed complex} if
      \(\dtwo\colon \mixed\to \mixed\) is a differential of degree \(-1\) with
      \([\done, \dtwo] := \done\dtwo + \dtwo\done = 0\).  
    \item A \textit{mixed DGA} is
      a mixed complex \((A, \done, \dtwo)\)
      together with a graded algebra structure on \(A\)
      such that \(\done\) and \(\dtwo\) are derivations with respect to it.
    \item A \textit{mixed differential graded (dg) Lie algebra} is
      a mixed complex \((\H, \done, \dtwo)\)
      together with a graded Lie algebra structure \([\ ,\ ]\) on \(\H\)
      such that \(\done\) and \(\dtwo\) are derivations with respect to \([\ ,\ ]\).
  \end{enumerate}
\end{defn}

Let  \((\mixed, \done, \dtwo)\) be a mixed complex.
We denote by \(\End(\mixed)\)
the endomorphism ring \(\Hom(\mixed, \mixed)\) of linear maps (of any degree). The ring \(\End(\mixed)\) is considered the Lie algebra with the bracket \([\ , \ ]\) defined by \([f, g] = fg - (-1)^{\deg{f}\deg{g}}gf\) for \(f\) and \(g \in \End( \mixed)\). We observe that \(\End(\mixed)\) is endowed with a dg Lie algebra structure whose differential is defined by \([d, \text{--} ]\) with the bracket and the differential $d$ of $\mixed$.
We see that a triple  \(( \End(\mixed), [\done, \ ],  [\dtwo,  \ ])\) is a mixed dg Lie algebra.
Moreover, for a DGA $A$, we define a differential graded  Lie subalgebra $\Der(A)$ of \(\End(A)\) consisting of derivations on $A$.
If \((A, \done, \dtwo)\) is a mixed DGA,
we observe that \(\Der(A)\) is a mixed dg Lie subalgebra of \(\End(A)\).

We recall a {\it derived version} of the non-positive derivations. Let $A$ be a commutative differential graded algebra $A$ (CDGA for short).
Following Block and Lazarev \cite{B-L05}, the Andr\'e--Quillen cohomology $H^*_{AQ}(A)$ of $A$ for $* \leq 0$ is defined by
\[H^*_{AQ}(A): = H^*(\Der(QA, A))\cong H^*(\Der(QA, QA), [d_{QA}, \text{--} ])
\]
with a cofibrant replacement $(QA, d_{QA})$ of $A$ in the category of CDGAs; see \cite{BG}. We regard $H^*_{AQ}(A)$ as a Lie algebra with the Lie bracket
on $H^*(\Der(QA, QA))$. Here we may choose as $QA$ a Sullivan model of $A$; see \cite[Section 12]{FHT1} for a general theory of Sullivan algebras.

Let \((\wedge V, d)\) be a Sullivan algebra for which \(V^1=0\).
Then we define a mixed DGA \((\wedge V\otimes\wedge \overline{V}, d, s)\),
where \(s\) is the derivation of degree \((-1)\) defined by \(sv=\bar{v}\) and \(s\bar{v}=0\) for \(v \in V\) and the differential \(d\) is the unique extension of \(d\colon\wedge V\to\wedge V\) which satisfies the condition that \([d, s] = 0\).
For simplicity of notation, we write
\(\HochschildModel = \wedge V\otimes\wedge\overline{V}\)
together with a decomposition
\(\HochschildModel = \bigoplus_k\HochschildModel_{(k)}\) of complexes,
where \(\HochschildModel_{(k)} = \wedge V\otimes\wedge^{k}\overline{V}\).
Observe that \(H(\HochschildModel)\) is isomorphic to the Hochschild homology of \((\wedge V, d)\); see \cite[Theorem 2.4(ii)]{BV88}.

Let $X$ be a simply-connected space whose rational cohomology \(H^*(X; \Q)\) is of finite type; that is,
\(\dim H^i(X; \Q)< \infty\) for each \(i\geq 0\).
Let \(LX\) be the free loop space which is the space of maps from \(S^1\) to \(X\) endowed with compact-open topology.
Suppose that \((\wedge V, d)\) is a Sullivan model for \(X\). Then
the complex \(\HochschildModel\) mentioned above is a Sullivan model for \(LX\); see \cite{V-S}.

We recall Sullivan's isomorphism $\Phi$ described in \cref{thm:CartanCal}.
Consider a sequence of the homotopy sets
  \[
    \xymatrix@C15pt@R18pt{
      \pi_n(\aut X) \ar[r]^k & [S^n \times X, X]  \ar[r]^-{\mu} &  [{\mathcal M}_X, {\mathcal M}_{S^n\times X}],
    }
  \]
where ${\mathcal M}_Y$ denotes a minimal Sullivan model for a space $Y$ and $\mu$ assigns a map $f$ a Sullivan representative for $f$.
We may replace ${\mathcal M}_{S^n\times X}$ with the DGA $H^*(S^n;\Q)\otimes {\mathcal M}_X$. 
Then we write
  \[
  (\mu\circ k)(\theta) = 1\otimes 1_{{{\mathcal M}_X}} + \iota \otimes \theta',
  \]
  where $\iota$ is the generator of  $H^n(S^n;\Q)$. Then, Sullivan's isomorphism $\Phi$ of Lie algebras
  \[
  \Phi : \pi_*(\aut X)\otimes \Q \to  H^{-*}(\Der({\mathcal M}_X), [d, \text{--}] ) = H^{-*}_{AQ}(A_{PL}^*(X))
  \]
  is defined by $\Phi(\theta) = \theta'$; see \cite{S} and also \cite[Theorems 3.6 and 4.3]{FLS}.

\section{Algebraic Cartan calculi}\label{sect:AlgCartanCal}

The homotopy Cartan calculus due to Fiorenza and Kowalzig \cite{FK} provides a systematic way to endow the shifted homology of a mixed complex $MC$ with the Batalin-Vilkovisky algebra structure and to give the Chas-Sullivan-Menichi \cite{CS, Luc2004} bracket to
the negative cyclic homology of $MC$; see \cite[Theorem D]{FK}.
Thus, it is crucial to consider examples of such a homotopy calculus.

In this section, we recall  the homotopy Cartan calculus with a slight generalization.
Roughly speaking, the calculus consists of two operations ($e$ and $L$) between complexes  and two {\it homotopies} ($S$ and $T$) between the two operations.
We give examples of the calculi by using a Sullivan model of the free loop space of a simply-connected space and the Hochschild chain complex of a differential graded algebra (DGA).
While the operations of the two homotopy calculi are identified on the homology level if a given DGA is a Sullivan algebra; see \cref{thm:e_and_L},
the difference between the homotopy calculi appears in the homotopy between operations; see \cref{prop:CartanCalculusSullivanModel} and \cref{prop:HCartan_Hochschild}.

\subsection{Homotopy Cartan calculus with slight generalization}\label{sect:Cartan}


Let \((\G, \dG)\) be a chain complex
and \((\H, \done, \dtwo)\) a mixed complex.

\begin{defn}[{cf.\ \cite[Definition 3.1]{FK}}]\label{defn:preCartan_generalization} A tuple
\((\G, \H, \cce, \ccl, \ccs)\) consisting of
  linear maps \(\cce, \ccl, \ccs\colon \G\to\H\) of degrees 1, 0 and \(-1\), respectively, is
  a \textit{homotopy pre-Cartan calculus}
  if the equalities
  \begin{align}
    &\ccl_\theta = \dtwo(\cce_\theta) + \done(\ccs_\theta) + \ccs_{\dlie \theta}, \\
    &\done(\cce_\theta) + \cce_{\dlie \theta} = 0 \ \ \text{and}  \\
    &
    \dtwo(\ccs_\theta) = 0
  \end{align}
hold for any \(\theta\in\G\).  The linear maps \(\cce\) and \(\ccl\) are called a \textit{contraction operator} (or \textit{cap product}) and a \textit{Lie derivative}, respectively.
\end{defn}

The first two conditions imply that
\(\cce\) and \(\ccl\) are chain maps of degree \(1\) and \(0\), respectively.

\begin{defn}\label{defn:Cartan_generalization}[{cf.\ \cite[Definition 3.7]{FK}}]
  Let \((\G, \dG, [ \ , \ ])\) be a dg Lie algebra and
  \((\H, \done, \dtwo, [ \ , \ ])\) a mixed Lie algebra; see \cref{defn:mixedObjects}.
  A \textit{homotopy Cartan calculus} \((\G, \H, \cce, \ccl, \ccs, \cct)\)
  is a homotopy pre-Cartan calculus \((\G, \H, \cce, \ccl, \ccs)\) equipped with
  a linear map \(\cct\colon\G\otimes\G\to\H\) satisfying the following equalities
  for any \(\theta, \rho\in\G\)
  \begin{align}
    &[\cce_\theta, \ccl_\rho] - \cce_{[\theta, \rho]} = \done(\cct_{\theta, \rho}) - \cct_{\dlie \theta, \rho} - (-1)^{\deg{\theta}}\cct_{\theta, \dlie \rho}, \\
    &[\ccs_\theta, \ccl_\rho] - \ccs_{[\theta, \rho]} = \dtwo(\cct_{\theta, \rho}).
  \end{align}
  Here \(\cct\) is called a \textit{Gelfan'd-Daletski\u\i-Tsygan homotopy}.
\end{defn}

\begin{rem}
  These definitions are equivalent to \cite[Definitions 3.1 and 3.7]{FK}
  if \((\H, \done, \dtwo, [ \ , \ ])\) is
  the tuple \((\End(\mixed), [\done', \ ],  [\dtwo',  \ ], [\ , \ ])\)
  which is given by a mixed complex \((\mixed, \done', \dtwo')\).
  In this case,
  we may call the calculus a homotopy Cartan calculus on the mixed complex \(\mixed\).
\end{rem}

The following is one of fundamental properties of a homotopy Cartan calculus.

\begin{lem}[{cf.\ \cite[Lemmas 3.4 and 3.10]{FK}}]\label{lem:Lie_homo}
  Let \((\G, \H, \cce, \ccl, \ccs, \cct)\) be a homotopy Cartan calculus.
  Then the map \(\ccl\colon \G\to \H\) is a morphism of dg Lie algebras.
\end{lem}

In particular, we see that a homotopy Cartan calculus \((\G, \H, \cce, \ccl, \ccs, \cct)\) gives a \((H(\G), [\ , \ ])\)-module structure to \(H(\H)\) via the map
\(H(L) : H(\G) \to H(\H)\). Moreover, it follows that \(H(e) :  H(\G) \to H(\H)\) is a morphism of \((H(\G), [\ , \ ])\)-modules.

If the linear map \(T\) in a homotopy Cartan calculus is trivial, then the map \(\cce\colon\G\to\H\) is regarded as a morphism of \((\G, \dG, [ \ , \ ])\)-modules, where the \(\G\)-module structure of \(\H\) is given by the morphism \(L\) of Lie algebras. We observe that our examples of homotopy Cartan calculi are in such a case; see Propositions  \ref{prop:CartanCalculusSullivanModel} and \ref{prop:HCartan_Hochschild} below.

\subsection{Homotopy Cartan calculus on the Sullivan model of free loop spaces}
\label{sect:CartanOnSullivan}
In this section, we give a homotopy Cartan calculus induced by a Sullivan algebra.
We recall the CDGA \(\HochschildModel\) described in Section \ref{sect:preliminaries}.

\begin{defn}\label{defn:derivationSullivan}
  For a derivation \(\theta\) on \(\wedge V\),
  we define derivations \(L_\theta\) and \(e_\theta\) on \(\HochschildModel\) by
  \begin{align}
    &L_\theta{v} = \theta v,\quad L_\theta{\bar{v}} = (-1)^{\deg{\theta}}s\theta v,\\
    &e_\theta{v} = 0,\quad e_\theta{\bar{v}} = (-1)^{\deg{\theta}}\theta v
  \end{align}
  for \(v\in V\).
  This defines linear maps
  \(L\colon \Der(\wedge V)\to \Der(\HochschildModel)\)
  of degree \(0\) and
  \(e\colon \Der(\wedge V)\to \Der(\HochschildModel)\)
  of degree \((-1)\).
\end{defn}

These derivations are introduced in \cite[Proof of Theorem 2.21]{KNWY}
by modifying constructions in \cite[Proposition 5]{Vigue1994}.

\begin{prop}
  \label{prop:CartanCalculusSullivanModel}
  The above maps give a homotopy Cartan calculus of the form
  \((\Der(\wedge V), \Der(\HochschildModel), e, L, \ccs=0, \cct=0)\).
\end{prop}
\begin{proof}
  Since \(\ccs=\cct=0\),
  we can reduce the equalities in \cref{defn:preCartan_generalization} and \cref{defn:Cartan_generalization} to
  \begin{align}
    &\ccl_\theta = [s, \cce_\theta], \\
    &[\done, \cce_\theta] + \cce_{\dlie \theta} = 0, \\
    &[\cce_\theta, \ccl_\rho] - \cce_{[\theta, \rho]} = 0.
  \end{align}
  A straightforward computation enables us to deduce that the equalities above hold on 
  \(V\oplus\overline{V}\).
  Since \(\cce_\theta\) and \(\ccl_\theta\) are derivations for any \(\theta\in\Der(\wedge V)\), we have the result.
\end{proof}

\subsection{Homotopy Cartan calculus on the Hochschild chain complex}
\label{sect:CartanOnHochschild}


In this section, we consider a homotopy Cartan calculus on the Hochschild chain complex of an augmented DGA.
While the domain of our calculus is restricted to the derivation ring of a DGA, the calculus is regarded as a DGA version of a homotopy Cartan calculus on the Hochschild chain complex of an associative algebra described in \cite[Example 3.13]{FK}.

We recall the Hochschild complex $C_*(A)$ of a DGA $(A, d)$ mentioned in Section \ref{sect:preliminaries}.
Then,
Connes' $B$ operator $B : C_* (A) \to C_*(A)$ is defined by
\[
  B_n := B | _{A\otimes T^n (s\bar{A}) } = s\circ (1 + t_n + \cdots + t^n_n)
\]
for $n\geq 0$.
Here, $t_n : A\otimes T^n (s\bar{A}) \to A\otimes T^n (s\bar{A})$ and $s : A\otimes T^n (s\bar{A}) \to A\otimes T^{n+1} (s\bar{A})$ are morphisms given by $t_0 =1$ and
\begin{align}
  t_n (a_0 [a_1 | \cdots |a_n]) &= (-1)^{|sa_n|(\varepsilon_n +1)} a_n [a_0| \cdots | a_{n-1}],
  \\
  s(a_0 [a_1 | \cdots |a_n]) &= 1[a_0 | a_1 | \cdots | a_n],
\end{align}
where $\varepsilon_i$ is the notation described in Section \ref{sect:preliminaries}.
Then, it follows from \cite[Example 1]{BV88} that the triple $(C_*(A), d, B)$ is a mixed complex.

Let $A'$ be an augmented DGA and  $\varphi : A \to A'$ a morphism of DGAs. For a derivation $\theta \in \Der (A,A')$, we define $L_{\theta} : C_*(A) \to C_*(A')$ by $L_{\theta} = \sum_ {i} L_{\theta,i}$ and
\[
  L_{\theta, i}(a_0 [a_1 | \cdots |a_n]) = \left\{
    \begin{array}{ll}
      \theta( a_0 )[\varphi(a_1) | \varphi(a_2) |\cdots | \varphi(a_n)] & ( i = 0),\\
      (-1)^{|\theta|(\varepsilon_i + 1)}\varphi(a_0) [\varphi(a_1) | \cdots | \theta (a_i) | \cdots | \varphi(a_n)] & (1 \leq i \leq n).
    \end{array}
  \right.
\]
We also define $e_{\theta} : C_*(A) \to C_*(A')$ by $e_{\theta}|_A =0$ and
\[
  e_{\theta}(a_0 [a_1 | \cdots |a_n])
  =
  (-1)^{|\theta||a_0| + |\theta| + |a_0|} \varphi(a_0) \theta (a_1)[ \varphi(a_2)|\cdots |\varphi(a_n)].
\]
Let $e'_{\theta}$ be the element in the Hochschild cochain complex $C^*(A;A')$ given by
\[
  e'_{\theta}(a_0 [a_1 | \cdots |a_n ] a_{n+1})
  = \left\{
    \begin{array}{ll}
      (-1)^{|\theta||a_0| + |\theta| + |a_0|} \varphi(a_0) \theta (a_1) \varphi(a_2) & ( n =1 ), \\
      0 & (n \neq 1).
    \end{array}
  \right.
\]
Then we see that $e_{\theta} = e'_{\theta} \cap \text{--}$, where the right-hand side is the cap product with $e'_{\theta}$
; see \cite[\S 3]{Luc2011}
for the cap product.
Moreover, we define $S_{\theta} : C_*(A) \to C_* (A')$ by $S_{\theta} |_A = 0$ and, for $n\geq 1$,
\[
  S_{\theta}|_{A\otimes T^n (s\bar{A})} = \sum_{j=1}^n \left( \sum_{k=0}^{n-j} s\circ t_n^k \right) \circ L_{\theta , j}.
\]

\begin{prop}\label{prop:HCartan_Hochschild}
  Let $e$, $L$ and $S$ be the morphisms described above.

  \begin{enumerate}
    \item\label{item:HCartan_Hochschild-1} The tuple $(\Der (A,A'), \Hom (C_*(A), C_*(A')), e, L, S)$ is a homotopy pre-Cartan calculus.
    \item\label{item:HCartan_Hochschild-2} The tuple $(\Der (A), \End(C_*(A)), e, L, S, T = 0)$ is a homotopy Cartan calculus on the mixed complex $(C_*(A), d, B)$.
  \end{enumerate}
\end{prop}

\begin{proof}
In order to prove \cref{item:HCartan_Hochschild-1}, it suffices to check the following equalities:
  \begin{enumerate}
    \item[(i)] \label{LBedS}
      $L_{\theta} = [B,e_{\theta}] + [d, S_{\theta}] + S_{\delta \theta}$,
    \item[(ii)] \label{de}
      $[d, e_{\theta}] + e_{\delta \theta} = 0$,
    \item[(iii)] \label{BS}
      $[B, S_{\theta}] = 0$.
  \end{enumerate}
  A straightforward computation allows us to deduce that
  \begin{align}
    &[d_1, S_{\theta}] = S_{\delta \theta},
    \\
    &L_{\theta} = (-1)^{|\theta|}e_{\theta} \circ B_n + (d_{2,0} + d_{2,n+1})\circ S_{\theta} \ \ \  \text{and}
    \\
    &B_{n-1} \circ e_{\theta}
      +
      \left( \displaystyle\sum_{i=1}^n  d_{2,i} \right) \circ S_{\theta }
      +
      (-1)^{|\theta|} S_{\theta} \circ d_2
      =
      0.
  \end{align}
  By combining the equalities, we obtain the formula (i).
  The linear maps $d_{1,i}$, $d_{2,i}$ and $e_{\theta}$ satisfy the followings relations:
  \begin{align}
    & \label{d1e}
      d_{1,i}\circ e_{\theta} = \left\{
      \begin{array}{ll}
        (-1)^{|\theta|+1}e_{\theta}\circ (d_{1,0} + d_{1,1}) - e_{\delta \theta} & (i=0),\\
        (-1)^{|\theta|+1}e_{\theta}\circ d_{1,i+1}  & ( 1\leq i \leq n-1),
      \end{array}
                                                      \right.
    \\
    & 
      d_{2,i}\circ e_{\theta} = \left\{
      \begin{array}{ll}
        (-1)^{|\theta|+1}e_{\theta}\circ (d_{2,0} + d_{2,1}) & (i=0),\\
        (-1)^{|\theta|+1}e_{\theta}\circ d_{2,i+1} & ( 1 \leq i \leq n-1).
      \end{array}
                                                     \right.
  \end{align}
  Then, we have the formula (ii) by combining the equalities (\ref{d1e}).
  Since $s\circ s =0$, $t_n \circ s =0$ and
  \begin{equation}\label{Ls}
    L_{\theta , i}\circ s = (-1)^{|\theta|}s\circ L_{\theta , i-1}
  \end{equation}
  for $i \geq 1$, it is immediate to verify the relation (iii). As a consequence, we have \cref{item:HCartan_Hochschild-1}. We consider the case where $A=A'$. In order to prove the assertion \cref{item:HCartan_Hochschild-2}, we show the following equalities
  \begin{itemize}
      \item[(iv)]\label{eL}
      $[e_{\theta}, L_{\rho}] - e_{[\theta, \rho]} = 0$ \ \ \ \text{and}
      \item[(v)]\label{SL}
      $[S_{\theta} , L_{\rho}] - S_{[\theta, \rho]} =0$.
  \end{itemize}
  Observe that $e_{\theta} = d_{2,0}\circ L_{\theta , 1}$. Moreover, we have
  \begin{equation}\label{Ld2}
    L_{\theta , i } \circ d_{2 , 0}
    = \left\{
      \begin{array}{ll}
        (-1)^{|\theta|} d_{2,0}\circ (L_{\theta , 0} + L_{\theta ,1} ) & (i=0), \\
        (-1)^{|\theta|} d_{2,0} \circ L_{\theta , i + 1} & ( i \geq 1)
      \end{array}
    \right.  \ \ \ \text{and}
  \end{equation}
  \begin{equation}\label{LL}
    L_{\theta , i } \circ L_{\rho , j}
    = \left\{
      \begin{array}{ll}
        L_{\theta \rho , i} & (i=j), \\
        (-1)^{|\theta||\rho|} L_{\rho , j} \circ L_{\theta , i} & ( i \neq j).
      \end{array}
    \right.
  \end{equation}
  The equalities \eqref{Ld2} and \eqref{LL} enable us to obtain the formula (iv).
  It is readily seen that 
  \begin{equation}\label{Lt}
    L_{\theta , i} \circ t_n= \left\{
      \begin{array}{ll}
        t_n \circ L_{\theta, n} & (i=0), \\
        t_n \circ L_{\theta, i-1}  & ( 1 \leq i \leq n).
      \end{array}
    \right.
  \end{equation}
  Therefore, we have the formula (v) by combining \eqref{Ls}, \eqref{LL} and \eqref{Lt}. 
\end{proof}

\subsection{Comparison among Cartan calculi}
\label{sect:naturality_of_CartanCalculus}
In this section, we compare two Cartan calculi
defined in \cref{sect:CartanOnSullivan} and \cref{sect:CartanOnHochschild}; see \cref{thm:e_and_L}.
We also show that a Sullivan model induces morphisms of graded Lie algebras
on the homology of Cartan calculi on the Hochschild complexes; see \cref{prop:relate_Cartan_calculi}.

As mentioned in the beginning of Section \ref{sect:AlgCartanCal},
the homotopy Cartan calculi of a Sullivan algebra in Proposition \ref{prop:CartanCalculusSullivanModel}
and the Hochschild complex in Proposition \ref{prop:HCartan_Hochschild} coincide with each other on homology. To see this,
we identify the Hochschild homology $HH_* (\SpaceModel)$ with the homology $H^*(\HochschildModel)$ by the quasi-isomorphism
$\Theta : C_*(\SpaceModel)  \to \HochschildModel$
defined by $\Theta (a_0[a_1| \cdots | a_n]) =
\frac{1}{n!} a_0sa_1\cdots sa_n$; see
\cite[Theorem 2.4]{BV88}.
Here $s$ is the unique derivation on $\HochschildModel$ stated in Section \ref{sect:preliminaries}.
We also recall the morphism
$\Theta' : {\mathcal L} \to  C_*(\SpaceModel)$ defined by
$ \Theta' (a_0sa_1\cdots sa_n)  = a_0\ast [a_1]\ast \cdots \ast[a_n]$, where $\ast$ denotes the shuffle product on the Hochschild complex; see \cite[Section 4]{G-J}.
Observe that $\Theta \circ \Theta' = 1$. Our main result in this section is described as follows.

\begin{thm}\label{thm:e_and_L}
With the same notation as above, one has a commutative diagram
  \[
    \xymatrix@C55pt@R20pt{
      H (\Der (\SpaceModel))
      \ar[r]^-{L_{( \ )}}_-{(\text{resp.} e_{( \ )})}
      \ar[dr]^-(0.7){L_{( \ )}}_-{(\text{resp.} e_{( \ )})}
      &
      \End (HH_*(\SpaceModel))
      \\
      & \Der (H^*(\HochschildModel)),
      \ar[u]_{i}
    }
  \]
  where $i$ is the monomorphism defined by the isomorphism $H(\Theta)$.
\end{thm}

\begin{proof}
  In order to prove the assertion, it suffices to show that the squares  
  \begin{equation}\label{diag:L_e}
    \xymatrix@C30pt@R15pt
    {
      C_*(\SpaceModel) \ar[d]_-{L_{\theta}}
      \ar[r]^-{\Theta}
      &
      \HochschildModel
      \ar[d]^-{L_\theta}
      \\
      C_*(\SpaceModel)
      \ar[r]_-\Theta
      &
      \HochschildModel,
    }
    \hspace{1em}
    \xymatrix@C30pt@R15pt
    {
      C_*(\SpaceModel)
      \ar[d]_-{e_{\theta}}
      &
      \HochschildModel
      \ar[l]_-{\Theta'}
      \ar[d]^-{e_\theta}
      \\
      C_*(\SpaceModel)
      \ar[r]_-\Theta
      &
      \HochschildModel
    }
  \end{equation}
  are commutative for $\theta \in \Der (\SpaceModel)$.
  Observe that $[L_{\theta}, s]=0$ in $\End (\HochschildModel)$ by Definition \ref{defn:derivationSullivan}.
  Then, we get
  \begin{align}
    &\Theta \circ L_{\theta}(a_0[a_1|a_2|\cdots |a_k])\\
    =&\dfrac{1}{k!}\left(
       \theta(a_{0})sa_1 sa_2 \cdots sa_k
       +\sum_{i=1}^k (-1)^{|\theta|(\varepsilon_i + 1)}a_0 sa_1 \cdots s\theta (a_i) \cdots sa_k
       \right)\\
    =&\dfrac{1}{k!}\left(
       \theta(a_{0})sa_1 sa_2 \cdots sa_k
       +\sum_{i=1}^k (-1)^{|\theta|\varepsilon_i}a_0 sa_1 \cdots L_{\theta}(sa_i) \cdots sa_k
       \right)
    \\
    =& L_{\theta} \circ \Theta (a_0[a_1|a_2|\cdots |a_k])
  \end{align}
  which implies the commutativity of the left-hand side square.
  On the other hand, given $a\bar{v}_1\bar{v}_2 \cdots \bar{v}_k \in \HochschildModel$ for $a\in \SpaceModel$ and $v_i \in V$.
  The induction on $k$ enables us to deduce that the shuffle product on $C_*(\SpaceModel)$ satisfies
  \[
    \Theta'( a\bar{v}_1\bar{v}_2 \cdots \bar{v}_k)
    =
    a * [v_1] * [v_2] * \cdots * [v_k]
    = \sum_{\sigma \in \mathfrak{S}}(-1)^{\varepsilon(\sigma)}a[v_{\sigma(1)}|v_{\sigma(2)}|\cdots | v_{\sigma(k)}],
  \]
  where $\mathfrak{S}_k$ is the symmetric group of degree $k$ and
  $(-1)^{\varepsilon (\sigma)}$ is the Koszul sign defined by the equality
  $(-1)^{\varepsilon (\sigma)} \bar{v}_{\sigma(1)} \bar{v}_{\sigma(2)} \cdots \bar{v}_{\sigma(k)} = \bar{v}_1\bar{v}_2 \cdots \bar{v}_k$ in $\HochschildModel$. Thus we have
  \begin{align}
    &\Theta \circ e_{\theta} \circ \Theta' (a\bar{v}_1\bar{v}_2 \cdots \bar{v}_k)
    \\
    =& \sum_{\sigma \in \mathfrak{S}_k}
       (-1)^{\varepsilon(\sigma) +  |\theta||a|+|\theta|+|a|}
       \dfrac{1}{(k-1)!}a\theta(v_{\sigma (1)}) \bar{v}_{\sigma(2)} \cdots \bar{v}_{\sigma(k)}
    \\
    =& \sum_{i=1}^k
       (-1)^{ |\theta||a|+|\theta|+|a| + |\bar{v}_i|(|\bar{v}_1| + \cdots + |\bar{v}_{i-1}|)} a\theta(v_i)\bar{v}_1 \cdots \bar{v}_{i-1}\bar{v}_{i+1} \cdots \bar{v}_k
    \\
    =& \sum_{i=1}^k
       (-1)^{ |\theta||a|+|a| + (|\theta|+1)(|\bar{v}_1| + \cdots + |\bar{v}_{i-1}|)} a\bar{v}_1 \cdots \bar{v}_{i-1}  e_{\theta}(\bar{v}_i) \bar{v}_{i+1} \cdots \bar{v}_k
   \\
    =& e_{\theta} (a\bar{v}_1\bar{v}_2 \cdots \bar{v}_k ).
  \end{align}
  This yields the commutativity of the right-hand side square.
\end{proof}




Let \(\varphi\colon (\wedge V, d)\xrightarrow{\simeq} (A, d)\) be
a (not necessarily connected) Sullivan model
of an augmented CDGA \((A, d)\).
In the rest of this section,
we show that the quasi-isomorphism $\varphi$ induces a morphism of Lie algebras between the homology Lie algebras of derivations. Moreover, we relate two homotopy Cartan calculi
\[
(\Der(A), \End(C_*(A)), e, L, S, 0) \ \ \text{and} \ \ (\Der(\wedge V), \End(C_*(\wedge V)), e, L, S, 0).
\]
We refer the reader to \cref{prop:HCartan_Hochschild} \cref{item:HCartan_Hochschild-2} for the calculi.

\begin{prop}
  \label{prop:relate_Cartan_calculi}
  There exist a homomorphism \(H^*(\Der(A))\to H^*(\Der(\wedge V))\)
  and an isomorphism  \(H^*(\End(C_*(A)))\xrightarrow{\cong} H^*(\End(C_*(\wedge V)))\)
  of graded Lie algebras such that the following diagrams commute:
  \begin{equation}
    \xymatrix@C20pt@R20pt{
      H^*(\Der(\wedge V)) \ar[r]^-{H(L)} & H^*(\End(C_*(\wedge V))) \\
      H^*(\Der(A)) \ar[r]^-{H(L)} \ar[u] & H^*(\End(C_*(A))), \ar[u]^-{\cong}
    }
    \xymatrix@C20pt@R20pt{
      H^*(\Der(\wedge V)) \ar[r]^-{H(e)} & H^{*+1}(\End(C_*(\wedge V))) \\
      H^*(\Der(A)) \ar[r]^-{H(e)} \ar[u] & H^{*+1}(\End(C_*(A))). \ar[u]^-{\cong}
    }
  \end{equation}
\end{prop}

\begin{proof}
  First we prove the proposition in the case where
  \(\varphi\colon (\wedge V, d)\to (A, d)\) is a \textit{surjective} quasi-isomorphism.
  The morphism \(\varphi\) of CDGAs gives rise to a commutative diagram
  \begin{equation}
    \xymatrix@C30pt@R20pt{
      H(\Der(\wedge V)) \ar[r]^-{H(L), H(e)} \ar[d]^-{\cong}_-{\varphi_*} & H(\End(C_*(\wedge V))) \ar[d]_-{\cong}^-{\varphi_*} \\
      H(\Der(\wedge V, A)) \ar[r]^-{H(L), H(e)} & H(\End(C_*(\wedge V), C_*(A))) \\
      H(\Der(A)) \ar[r]^-{H(L), H(e)} \ar[u]^-{\varphi^*} & H(\End(C_*(A))). \ar[u]^-{\cong}_-{\varphi^*}
    }
  \end{equation}
  Since the functor \(C_*(-)\) preserves quasi-isomorphisms, it follows that
  the right vertical maps are isomorphisms.
  \cref{cor:Der_quasi-iso} \cref{item:Der_quasi-iso_codomain} implies that
  the upper left map \(\varphi_*\) is an isomorphism.
  Now we need to prove that the two vertical composites are morphisms of Lie algebras.
  Since the right one can be proved similarly to (and easier than) the left one,
  we give only a proof for the left one.

  It follows from \cref{prop:lambda_map} that the map \(\varphi_* \colon
  \Der(\wedge V) \to \Der(\wedge V, A)\) is a surjective quasi-isomorphism. Then, for any elements \([f], [g] \in H(\Der(A))\),
  there are cocycles \(f', g'\in\Der(\wedge V)\) such that
  \(f\varphi=\varphi f'\) and \(g\varphi=\varphi g'\).
  Therefore, we see that
  \((\varphi_*)^{-1}\circ\varphi^*[f] = [f']\) and \((\varphi_*)^{-1}\circ\varphi^*[g] = [g']\).
  By a straightforward computation, we have \([f, g]\varphi = \varphi[f', g']\)
  and this completes the proof for the particular  case.

  Next we deal with a general case.
  To this end, the ``surjective trick'' is applicable. In fact,
  the map \(\varphi\colon (\wedge V, d)\to (A, d)\) factors as
  \((\wedge V, d)\xrightarrow[\simeq]{i}(\wedge V, d)\otimes (E(A), \delta)\xrightarrow[\simeq]{\varphi'}(A, d)\) with
   a contractible algebra \((E(A), \delta)\) 
   and
  the canonical maps \(i\) and \(\varphi'\); see \cite[Section 12 (b)]{FHT1} for details.
  Now we have a homotopy inverse
  \(r\colon (\wedge W, d)\to (\wedge V, d)\) of \(i\) defined by sending \(A\) to zero,
  where  \((\wedge W, d) = (\wedge V, d)\otimes (E(A), \delta)\).
  Observe that \(W = V\oplus A\oplus \delta A\).
  Then \(r\) and \(\varphi'\) are \textit{surjective} quasi-isomorphisms
  and hence the first half of the proof gives the following diagram:
  \begin{equation}
    \xymatrix@C30pt@R20pt{
      H(\Der(\wedge V)) \ar[r]^-{H(L), H(e)} \ar[d]^-{\cong}_-{r^*} & H(\End(C_*(\wedge V))) \ar[d]_-{\cong}^-{r^*} \\
      H(\Der(\wedge W, \wedge V)) \ar[r]^-{H(L), H(e)} & H(\End(C_*(\wedge W), C_*(\wedge V))) \\
      H(\Der(\wedge W)) \ar[r]^-{H(L), H(e)} \ar[d]^-{\cong}_-{\varphi_*} \ar[u]_-{\cong}^-{r_*} & H(\End(C_*(\wedge W))) \ar[d]_-{\cong}^-{\varphi_*} \ar[u]^-{\cong}_-{r_*}\\
      H(\Der(\wedge W, A)) \ar[r]^-{H(L), H(e)} & H(\End(C_*(\wedge W), C_*(A))) \\
      H(\Der(A)) \ar[r]^-{H(L), H(e)} \ar[u]^-{\varphi^*} & H(\End(C_*(A))). \ar[u]^-{\cong}_-{\varphi^*}
    }
  \end{equation}
  The upper left map \(r^*\) is an isomorphism with the inverse $i^*$ by \cref{cor:Der_quasi-iso} \cref{item:Der_quasi-iso_domain}.
  This completes the proof of the general case.
\end{proof}

\subsection{Injectivity of the contraction \(e\) on homology}\label{sect:Faithfulness_e}
Let \((\wedge V, d)\) be a simply-connected Sullivan algebra
whose homology satisfies the Poincar\'e duality. We assume that the fundamental class is in
\(H^m(\wedge V)\).
In this section,
we study properties of the derivation
\(H(e_\theta)\colon H(\HochschildModel)\to H(\HochschildModel)\)
and show injectivity of this map. We recall the differential graded module \(\HochschildModel_{(k)}\) defined in \cref{sect:CartanOnSullivan}.

\begin{thm}
  \label{thm:e_fundamental_class} Let
  \((\Der(\wedge V), \Der(\HochschildModel), e, L, 0, 0)\) be the homotopy Cartan calculus in \cref{prop:CartanCalculusSullivanModel}.
  For any \([\theta]\neq 0\in H^{-n-1}(\Der(\wedge V))\),
  there exists a cohomology class \([\alpha]\in H^{m+n}(\HochschildModel_{(1)})\) such that
  \(H(e_\theta)[\alpha]\) is the same as the fundamental class in
  \(H^m(\wedge V)\subset H^m(\HochschildModel)\).
\end{thm}

This theorem immediately implies the following corollary.

\begin{cor}
  \label{cor:e_faithful}
  Suppose that  \(H^*(\wedge V)\) satisfies the Poincar\'e duality. Then
  the map
  \(H^*(\Der(\wedge V))\to \Der^{*+1}(H(\HochschildModel))\)
  induced by the contraction
  \(e\colon \Der^*(\wedge V)\to\Der^{*+1}(\HochschildModel)\)
  is a monomorphism and hence so is
  the map
  \(H^*(e)\colon H^*(\Der(\wedge V))\to H^{*+1}(\Der(\HochschildModel))\).
\end{cor}

Thanks to \cref{cor:e_faithful} and \cref{prop:Lie_e} below, we have \cref{thm:e-inj}.
It is expected that the contraction operator $e$ is injective for more general spaces.

The rest of this section is devoted to proving  \cref{thm:e_fundamental_class}.
Let \((\wedge V, d)_*\) be the linear dual of \((\wedge V, d)\) and
\(D\colon (\wedge V, d)\to (\wedge V, d)_*\) the duality map;
that is, a quasi-isomorphism of \((\wedge V, d)\)-modules of degree \((-m)\)
given by the cap product with the representing cycle of the fundamental class.
For a (non-negative) integer \(k\), define a quasi-isomorphism \(\pd\) by the composition
\begin{equation}
  \pd\colon
  \Hom_{\wedge V}(\HochschildModel_{(k)}, \wedge V)
  \xrightarrow[\simeq]{D_*} \Hom_{\wedge V}(\HochschildModel_{(k)}, (\wedge V)_*)
  \xrightarrow[\cong]{\mathrm{adjoint}}\Hom(\HochschildModel_{(k)}, \Q)
  =(\HochschildModel_{(k)})_*
\end{equation}
and denote its adjoint by
\(\ad(\pd)\colon \Hom_{\wedge V}^{-n}(\HochschildModel_{(k)}, \wedge V)\otimes\HochschildModel_{(k)}^{m+n}\to\Q\).

By a straightforward computation, we have

\begin{lem}
  \label{lem:pd_and_D}
  For any integer \(n\) and \(k\), we have a commutative diagram
  \begin{equation}
    \xymatrix@C20pt@R20pt{
      \Hom_{\wedge V}^{-n}(\HochschildModel_{(k)}, \wedge V)\otimes \HochschildModel_{(k)}^{m+n} \ar[d]^-{\ev} \ar[r]^-{\ad(\pd)} & \Q \\
      (\wedge V)^m \ar[r]^-{D} & ((\wedge V)_*)^0. \ar[u]_-{\ev_1}^-{\cong}
    }
  \end{equation}
\end{lem}

\begin{prop}
  \label{prop:non_deg_pairing}
  Let \(n\) and \(k\) be integers.
  \begin{enumerate}
    \item The pairing
      \begin{equation}
        H(\ev)\colon H^{-n}(\Hom_{\wedge V}(\HochschildModel_{(k)}, \wedge V))
        \otimes H^{m+n}(\HochschildModel_{(k)})
        \to H^m(\wedge V) \cong \Q
      \end{equation}
      is non-degenerate.
    \item\label{item:eval_to_fundamental_class} For any
      \([f]\neq 0\in H^{-n}(\Hom_{\wedge V}(\HochschildModel_{(k)}, \wedge V))\),
      there is a cohomology class
      \([\alpha]\in H^{m+n}(\HochschildModel_{(k)})\)
      such that \([f(\alpha)] \in H^m(\wedge V)\) is the same as the fundamental class.
  \end{enumerate}
\end{prop}
\begin{proof}
  Since \(D\) induces an isomorphism on homology,
  \cref{lem:pd_and_D} identifies \(H(\ev)\) and \(H(\ad(\pd))\) up to isomorphism.
  Hence the proposition follows from the fact that
  \(H(\pd)\) is an isomorphism.
\end{proof}

\newcommand{\eFromHom}{\tilde{e}}
In order to prove \cref{thm:e_fundamental_class}, we
represent the contraction \(e\) as a composite of maps. \cref{prop:lambda_map} below asserts that the  linear map
\(\lambda\colon \Der(\wedge V)\to \Hom_{\wedge V}(\HochschildModel_{(1)}, \wedge V)\cong\Hom(\overline{V}, \wedge V)\) defined by
\(\lambda(\theta)(\bar{v}) = (-1)^{\deg{\theta}}\theta(v)\)
for \(v\in V\) and \(\theta\in\Der(\wedge V)\) is an isomorphism of complexes of degree \(1\).
Moreover, we define a chain map
\(\eFromHom\colon \Hom_{\wedge V}(\HochschildModel_{(1)}, \wedge V)\to\Der(\HochschildModel)\)
of degree 0 by
\(\eFromHom(f)(v)=0\) and \(\eFromHom(f)(\bar{v})=f(\bar{v})\)
for \(v\in V\) and \(f\in\Hom_{\wedge V}(\HochschildModel_{(1)}, \wedge V)\).
Then we have a commutative diagram
\begin{equation}
  \xymatrix@C20pt@R20pt{
    \Der(\wedge V) \ar[r]^-{e} \ar[d]_-{\cong}^-{\lambda} & \Der(\HochschildModel) \\
    \Hom_{\wedge V}(\HochschildModel_{(1)}, \wedge V). \ar[ru]^-{\eFromHom}
  }
\end{equation}

\begin{proof}[Proof of \cref{thm:e_fundamental_class}]
  Take an element \([\theta]\neq 0\in H^{-n-1}(\Der(\wedge V))\).
  Since \(\lambda\) is an isomorphism, we have
  \(H(\lambda)[\theta]\neq 0\in H^{-n}(\Hom_{\wedge V}(\HochschildModel_{(1)}, \wedge V))\).
  Hence, by \cref{prop:non_deg_pairing} \cref{item:eval_to_fundamental_class},
  there exists a cohomology class \([\alpha]\in H^{m+n}(\HochschildModel_{(1)})\) such that
  \([\lambda(\theta)(\alpha)]\in H^m(\wedge V)\) is the same as the fundamental class.
  Then we see that \([(\eFromHom\lambda(\theta))(\alpha)]\in H^m(\HochschildModel)\) is nothing but
  the fundamental class. Therefore, the theorem follows from the above commutative diagram.
\end{proof}

\section{Geometric counterparts of Cartan calculi}
\label{sect:GeometricDiscription}

In this section, we assume
that the underlying field is of arbitrary characteristic.
Theorem \ref{thm:e_and_L} asserts that the operations $L$ and $e$ appeared in Proposition  \ref{prop:CartanCalculusSullivanModel} coincide with the Lie representation and the contraction in Proposition \ref{prop:HCartan_Hochschild} on homology, respectively. In this section, we consider geometric constructions of the operations $L$ and $e$ on homology.
Moreover, as mentioned in Introduction, a geometric description of the isomorphism $\Phi$ of Lie algebras in Theorem \ref{thm:CartanCal} is given.

\subsection{Geometric descriptions of the operations  $L$ and $e$}\label{sect:geom_L_e}

Given $\theta \in \pi_n (\aut X)$ which is represented by $\theta : S^n \to \aut X$.
Let $\ad (\theta) : S^n \times X \to X$ be the adjoint of $\theta$ and consider the map between the free loop spaces $L(\ad (\theta)) : LS^n \times LX \to LX$ defined by
$(L(\ad (\theta)(l, \gamma))(t) = \ad (\theta)(l(t), \gamma(t))$ for $(l, \gamma)\in LS^n\times LX$ and $t \in S^1$.
Let $\ev_0 : LS^n \to S^n$ be the evaluation map at $0$. We define $L : \pi_n (\aut X) \to \End^{-n} (H^* (LX))$ by the composite
\begin{equation}\label{defn:L_geom}
\xymatrix@C30pt@R30pt{
L_\theta : H^* (LX)
\ar[r]^-{L(\ad (\theta))^*}
&
H^*(LS^n \times LX)
\ar[r]^-{\int_{[S^n]}}
&
H^*(LX),
}
\end{equation}
where $\int_{[S^n]}$ denotes the integration along the image of the fundamental class of $S^n$ by the map $\ev_0^* : H^n(S^n) \to H^n(LS^n)$.
The rotation on $S^1$ induces the action $\mu : S^1\times LX \to LX$ on the free loop space. By definition, the BV-operator $\Delta$ on $H^*(LX)$ is the composite
\[
\xymatrix@C20pt@R20pt{
\Delta: H^*(LX) \ar[r]^{\mu^*} & H^*(S^1\times LX) \ar[r]^-{\int_{S^1}} & H^{*-1}(LX),
}
\]
where $\int_{S^1}$ is the integration along the fundamental class of $S^1$.
Let $\overline{[S^n]}$ be the cohomology class in $H^{n-1}(LS^n)$ which is the image of the fundamental class of $S^n$ by the composite
\[
\xymatrix@C20pt@R20pt{
H^n (S^n)
\ar[r]^-{\ev_0^*}
&
H^n (LS^n)
\ar[r]^-{\Delta}
&
H^{n-1}(LS^n).
}
\]
Then, we define a linear map $e : \pi_n (\aut X) \to \End^{-n+1} (H^* (LX))$ by the composite
\begin{equation}\label{defn:e_geom}
\xymatrix@C30pt@R30pt{
e_\theta : H^* (LX)
\ar[r]^-{L(\ad (\theta))^*}
&
H^*(LS^n \times LX)
\ar[r]^-{\int_{\overline{[S^n]}}}
&
H^*(LX).
}
\end{equation}

We first consider properties of the operation $L$.
\begin{thm}\label{thm:LieAlgL}
  \begin{enumerate}
    \item\label{item:LieAlgL-i} The map
      $L : \pi_*(\aut X) \to \Der^{-*}(H^*(LX))$
      is a morphism of Lie algebras.
    \item\label{item:LieAlgL-ii} For each $\theta$ in  $\pi_*(\aut X)$,
      the derivation $L_\theta$  commutes with the BV operator
      $\Delta : H^*(LX) \to H^{*-1}(LX)$.
  \end{enumerate}
\end{thm}

We postpone the proof to \cref{app:bar{L}}; see the argument after \cref{thm:LieAlg}.
The operation $L$ in (\ref{defn:L_geom}) is extended to that on the equivariant cohomology $H^*_{S^1}(LX)$.
\cref{thm:LieAlgL} is proved by considering a result on the equivariant version of the operator $L$.

\begin{rem}
One might expect that the same construction as that in (\ref{defn:L_geom}) and (\ref{defn:e_geom}) is applicable to {\it other} element in $H^*(LS^n)$, or more generally in $H^*(LS)$ for a simply-connected space $S$. In fact, it is possible. Moreover, we see that such operations incorporate $e$ and $L$ together with interesting properties, for instance, the Cartan formula such as that for Steenrod operations.
The topic will be discussed in more detail in \cite{NW2022}.
\end{rem}


In the rest of this section, we assume that the underlying field and the coefficients of cohomology rings of spaces are rational unless otherwise specified.

We relate Sullivan's isomorphism $\Phi$ mentioned in Section \ref{sect:preliminaries} with the Lie derivative  $L_{( \ )}$.

\begin{prop}\label{prop:LieAlg-New} Let $X$ be a simply-connected space of finite type and $\wedge V$ the minimal Sullivan model for $X$. Then there exists a commutative diagram
  \[
    \xymatrix@C20pt@R20pt{
      \pi_*(\aut X)\otimes \Q \ar[r]^-{L}  \ar[d]^{\cong}_{\Phi}& \Der^{-*}(H^*(LX; \Q)) \ar[d]^{\cong}\\
      H_*(\Der(\wedge V)) \ar[r]_-{L} &
      \Der^{-*}(H^*(\HochschildModel)).
    }
  \]
\end{prop}

We prove Proposition \ref{prop:LieAlg-New} in Appendix \ref{app:bar{L}} together with its equivariant version; see \cref{thm:LieAlg-New}.
Next we give a relationship between the operation $e$ and the isomorphism $\Phi$.

\begin{prop}\label{prop:Lie_e} Under the same assumption as in Proposition
\ref{prop:LieAlg-New}, there exists
a commutative diagram
\[
\xymatrix@C30pt@R20pt{
\pi_n (\aut X) \otimes \Q
\ar[r]^-{(-1)^ne}
\ar[d]_-{\Phi}^{\cong}
&
\End^{-n+1} (H^* (LX ; \Q))
\\
H_n ( \Der (\wedge V))
\ar[r]^-{e}
&
\Der^{-n+1} (H^* (\HochschildModel)).
\ar[u]_-{\text{a monomorphism}}
}
\]
\end{prop}

\begin{proof}
We first consider a rational model for $e_\theta$ described in \eqref{defn:e_geom}.
Let ${\mathcal M}_{S^n}$ be the Sullivan model for $S^n$ which is of the form
\[
{\mathcal M}_{S^n} = \left\{
\begin{array}{ll}
(\wedge (u), 0) & ( n : \text{odd}),\\
(\wedge (u , u'), du' = u^2) & (n : \text{even}),
\end{array}
\right.
\]
where $|u|=n$, $|u'|=2n-1$.
Let  $\HochschildModel_{S^n}$ be the Sullivan model for $LS^n$ induced by ${\mathcal M}_{S^n}$ and  $\varphi : \wedge V \to {\mathcal M}_{S^n}\otimes \wedge V$ a Sullivan representative for $\ad (\theta)$; see \cref{sect:preliminaries}.
It follows from Lemma \ref{lem:ModelLf} that a Sullivan representative $\HochschildModel \varphi : \HochschildModel \to  \HochschildModel_{S^n}\otimes \HochschildModel$ for $L(\ad (\theta))$ is given by
\begin{equation}\label{model:L(ad(f))}
\HochschildModel \varphi  (v) = \varphi (v)
\ \ \text{and} \ \
\HochschildModel \varphi  (\bar{v}) = s \varphi (v)
\end{equation}
for $v \in V$.
We define a morphism $\int_{\bar{u}} : \HochschildModel_{S^n} \to \Q$ of chain complexes of degree $-n+1$ by $\int_{\bar{u}}(\bar{u}) = 1$ and $\int_{\bar{u}}(w) = 0$ for bases $w$ with $w\neq \bar{u}$.
Since the cohomology class $\overline{[S^n]}$ is represented by $\bar{u}$, the definition of $e_\theta$ implies that the composite $(\int_{\bar{u}}\otimes 1) \circ \HochschildModel \varphi $ is a rational model for  $e_\theta$.
Now, we may write
\[
\varphi (v) \equiv 1\otimes v + u \otimes \theta' (v)
\]
for $v\in V$ modulo $({\mathcal M}_{S^n})^{>n} \otimes \wedge V$.
By definition, we see that $\Phi (\theta)(v) = \theta' (v)$.
Therefore, it follows from \eqref{model:L(ad(f))} that
\begin{align}
\HochschildModel \varphi (v)
&\equiv  1\otimes v + u \otimes \theta' (v), \\
\HochschildModel \varphi (\bar{v})
&\equiv  s (1\otimes v + u \otimes \theta' (v))
=1 \otimes \bar{v} + \bar{u} \otimes \theta' (v) + (-1)^n u\otimes s\theta' (v)
\end{align}
modulo $(\HochschildModel_{S^n})^{>n} \otimes \HochschildModel$.
We have $(\int_{\bar{u}}\otimes 1) \circ \HochschildModel \varphi (v)=0$ and $(\int_{\bar{u}}\otimes 1) \circ \HochschildModel \varphi (\bar{v}) = (-1)^n e_{\theta}(\bar{v})$. This completes the proof.  
\end{proof}

Finally we are ready to prove \cref{thm:CartanCal}.
\begin{proof}[Proof of \cref{thm:CartanCal}]
We observe that all results in this manuscript remain  true even if the underlying field ${\mathbb Q}$ is replaced with ${\mathbb R}$. In fact, after taking the tensor product $\text{--}\otimes_\Q{\mathbb R}$, we have these results. Moreover, we recall the fact that there exists a sequence of quasi-isomorphisms connecting $\Omega^*(M)$ with $A_{PL}(M)\otimes_\Q{\mathbb R}$, where $A_{PL}(M)$ denotes the CDGA of polynomial differential forms on a  manifold $M$; see \cite[Section 11 (d)]{FHT1} for the details. Therefore, a Sullivan minimal model $\wedge W \stackrel{\simeq}{\to} A_{PL}(M)$ gives rise to a minimal model $m : \wedge V \stackrel{\simeq}{\to} Q\Omega^*(M)$ in which $V = W\otimes_\Q{\mathbb R}$. Here $Q\Omega^*(M)$ denotes a cofibrant replacement of $\Omega^*(M)$ in the category $\mathcal{A}$ of CDGAs endowed with the model category structure described in \cite{BG}.

Let $M$ be a simply-connected manifold.
To prove Theorem \ref{thm:CartanCal},
we first construct an isomorphism between \(H_*(\Der(Q\Omega^*(M)))\) and \(H_*(\Der(\wedge V))\),
by applying the proof of  \cite[Theorem 2.8]{B-L05}.
With the notation above, the map $m$ has a factorization
\[
  \xymatrix@C20pt@R10pt{
  \wedge V \ar[rr]^-m_-\simeq \ar[rd]_-i^-\simeq&& Q\Omega^*(M) \\
& A \ar[ru]_-p&
}
\]
in $\mathcal{A}$, where $i$ is a trivial cofibration and $p$ is a fibration, namely an epimorphism. Observe that $p$ is also a quasi-isomorphism.
Therefore, we have a right splitting $g$ of $p$ with $p\circ g = id_{Q\Omega^*(M)}$ and a map $h_1:  \Der(A)\to \Der(Q\Omega^*(M))$ defined by $h_1(\theta) = p\circ \theta \circ g$. Moreover, since each object in $\mathcal{A}$ is fibrant, it follows that $i$ admits a left splitting $r : A \to \wedge V$ with  $r \circ i = id_{\wedge V}$.
Thus, a chain map $h_2 : \Der(A)\to \Der(\wedge V)$ is defined by $h_2(\theta)=r\circ \theta \circ i$.
As a consequence, we have a diagram consisting of commutative squares
\[
  \xymatrix@C20pt@R18pt{
   H_*(\Der(Q\Omega^*(M))) \ar[r]^-{L}_-{e}  & \End^{-*}(HH_*(\Omega^*(M))) \\
    H_*(\Der(A)) \ar[r]^-{L}_-{e}  \ar[d]^{\cong}_{(h_2)_*} \ar[u]_{\cong}^{(h_1)_*}& \End^{-*}(HH_*(A)) \ar[d]^{HH(r)\circ ( \ ) \circ HH(i)}_{\cong}
    \ar[u]_{HH(p)\circ ( \ ) \circ HH(g)}^{\cong}\\
    H_*(\Der(\wedge V)) \ar[r]^-{L}_-{e} &
    \End^{-*}(HH_*(\wedge V)).
  }
\]
We observe that left maps $(h_1)_*$ and $(h_2)_*$ are isomorphisms; see the proof of  \cite[Theorem 2.8 (1)]{B-L05}.

Finally Theorem \ref{thm:e_and_L}, Propositions  \ref{prop:LieAlg-New} and \ref{prop:Lie_e} enable us to obtain the commutative diagram in Theorem \ref{thm:CartanCal}.
\end{proof}

\subsection
{The map $\Gamma_1$ due to F\'elix and Thomas}
\label{sect:Gamma_1}

Throughout this section, we assume that $M$ is a simply-connected closed manifold of dimension $m$. Let $\SpaceModel = (\wedge V ,d)$ be the minimal Sullivan model for $M$ and $C_* (\SpaceModel)$ the Hochschild chain complex of $\SpaceModel$.
Recall the direct sum decomposition $\HochschildModel = \wedge V\otimes\wedge\overline{V} = \oplus_k \HochschildModel_{(k)}$ of complexes from \cref{sect:preliminaries}.
Thus we have decompositions $H^* (\HochschildModel)=\oplus_k H^*(\HochschildModel_{(k)})$ and  $H_*(LM; \Q) \cong \oplus_k H^{-*}( \Hom (\HochschildModel_{(k)} , \Q))$ which are called the Hodge decompositions.
We put
\[
  H_*^{(k)}(LM) := H^{-*}( \Hom (\HochschildModel_{(k)} , \Q)).
\]
Since the Hochschild cohomology $HH^*(\SpaceModel,A)$ with coefficients in a $\SpaceModel$-module $A$ is isomorphic to the homology of the complex $\Hom_{\SpaceModel}(\HochschildModel, A)$, we also have the direct sum decomposition
$HH^*(\SpaceModel,A) \cong \oplus_k  HH^*_{(k)}(\SpaceModel,A)$,
where
\[
  HH^*_{(k)}(\SpaceModel,A) = H^*(\Hom_{\SpaceModel}(\HochschildModel_{(k)}, A)).
\]
By a direct computation, we have
\begin{prop}\label{prop:lambda_map}
The map $\lambda : \Der (\SpaceModel, A) \to \Hom_{\SpaceModel} (\HochschildModel_{(1)}, A) \cong \Hom (\overline{V},A)$ of degree $1$ defined by
$\lambda (\theta)(\bar{v})=(-1)^{|\theta|}\theta (v)$
for $\theta \in \Der (\SpaceModel, A)$ and $\bar{v}\in \overline{V}$ is an isomorphism of complexes of degree $1$.
\end{prop}

\begin{cor}\label{cor:Der_quasi-iso}
  Let \(\eta : \wedge V \to A\) be a morphism of CDGAs,
  where \(\wedge V\) is a Sullivan algebra.
  \begin{enumerate}
    \item \label{item:Der_quasi-iso_codomain}
      For any quasi-isomorphism \(\varphi : A \to B\) of CDGAs,
      the map \(\varphi_* : H_*(\Der(\wedge V, A)) \to H_*(\Der(\wedge V, B))\)
      induced by \(\varphi\) is an isomorphism,
      where \(B\) in the codomain is regarded as a \(\wedge V\)-module
      via the composite \(\varphi\circ\eta\).
    \item \label{item:Der_quasi-iso_domain}
      For any quasi-isomorphism $\psi : \wedge W \to \wedge V$ of Sullivan algebras,
      the map $\psi^* : H_*(\Der(\wedge V, A))\to H_*(\Der(\wedge W, A))$
      induced by $\psi$ is an isomorphism,
      where $A$ in the codomain is regarded as a $\wedge W$-module
      via the composite $\eta\circ\psi$.
  \end{enumerate}
\end{cor}

\begin{proof}
The differential graded module $\HochschildModel_{(1)}$ is a semifree $\wedge V$-module. Then the results follow from \cite[Theorem 6.10]{FHT1} and the naturality of
$\lambda$ in Proposition \ref{prop:lambda_map}.
\end{proof}

Poincar\'e duality for manifolds gives rise to a duality between the direct summands $H_*^{(k)}(LM)$ and $HH^*_{(k)}(\SpaceModel,A)$.
To see this, let $A$ be an $m$-dimensional Poincar\'{e} duality model for a simply-connected manifold $M$ introduced in \cite{LS08} equipped with $\varphi : \SpaceModel \to A$ a quasi-isomorphism of CDGAs.
Denote by $A_* := \Hom (A,\Q)$ the linear dual of $A$ with the differential defined by $\alpha \mapsto -(-1)^{|\alpha|} \alpha \circ d_A$ for $\alpha \in A_*$, where $d_A$ denotes  the differential of $A$; see Remark \ref{rem:Sign_ModelEv}.
Let $\{ a_i \}_{i=1}^N$ be a homogeneous basis with $a_N=w_A$ a representative of the fundamental class of $M$. We denote by  $\{ a_i^* \}_{i=1}^N$ the dual basis. Let  $D_A : A \to A_*$ be the duality map; that is, $D_A$ is an isomorphism of $A$-modules defined by $D_A (a)(b) = \omega_A^*(ab)$.
Observe that $A$ is regarded as a $\SpaceModel$-module via  $\varphi$.

Put $\HochschildModel^A := A\otimes _{\SpaceModel}\HochschildModel \cong A\otimes \wedge\overline{V}$. We observe that $\HochschildModel^A$ is also a rational model for $LM$. The direct sum decomposition of $\HochschildModel$ mentioned above induces a decomposition $\HochschildModel^A = \oplus_k \HochschildModel^A_{(k)}$, where $\HochschildModel^A_{(k)}:=A\otimes _{\SpaceModel}\HochschildModel_{(k)}$.
Then, we have an isomorphism
\begin{eqnarray}\label{eq:PD}
  \xymatrix@C10pt@R20pt{
    \ \ \ \ \ \ \ \PD : HH^n_{(k)} (\SpaceModel)
    \ar[r]^-{\varphi_*}_-{\cong}
    &
    HH^n_{(k)} (\SpaceModel,A)
    \ar[r]^-{D_{A*}}_-{\cong}
    &
    HH^{n-m}_{(k)} (\SpaceModel, A_*)
    \ar[ld]_-{\text{adjoint}}^-{\cong}
    \\
    &
    H^{n-m}( \Hom (\HochschildModel^A_{(k)} , \Q))
    \ar[r]^-{(\varphi \otimes 1)^*}_-{\cong}
    &
    H_{-n+m}^{(k)}(LM).
  }
\end{eqnarray}

Let $\Omega\aut M_0$ be the connected component of the based loop space $\Omega \aut M$ containing the constant loop at $id\in \aut M$.
We here recall the morphism $\Gamma_1$ due to F\'elix and Thomas.
Let $g :\Omega \aut M_0 \times M \to LM$ be a map defined by
\[
g(\gamma ,x)(t) = \gamma (t)(x)
\]
for $\gamma \in \Omega \aut M_0$, $x \in M$ and $t\in S^1$.
In \cite{FT04}, F\'elix and Thomas show that the morphism $\Gamma_1$ defined by the composite
\[
  \xymatrix{
    \pi_n (\Omega\aut M_0) \otimes \Q
    \ar[r]^-{\text{Hur}}
    &
    H_n (\Omega\aut M_0;\Q)
    \ar[d]_-{\times [M]}
    \\
    &
    H_{n+m}(\Omega \aut M_0 \times M  ;\Q)
    \ar[r]^-{g_*}
    &
    H_{n+m}(LM  ;\Q)
  }
\]
is injective for $n\geq 1$ and that the image of $\Gamma_1$ is isomorphic to $H^{(1)}_{n+m}(LM)$, where $\text{Hur}$ denotes the Hurewicz map and $[M]\in H_m(M)$ is the fundamental class.

The main theorem in this section is as follows.
\begin{thm}\label{thm:Model2_aut} With the notation above,
  the diagram
  \[
    \xymatrix@C30pt@R20pt{
      \pi_n (\aut M) \otimes \Q
      \ar[rr]^-{\Phi}_-{\cong}
      &&
      H_n (\Der (\SpaceModel))
      \ar[d]^-{\lambda}_-{\cong}
      \\
      \pi_{n-1} (\Omega\aut M_0) \otimes \Q
      \ar[r]^-{\Gamma_1}_-{\cong}
      \ar[u]^-{\partial}_-{\cong}
      &
      H_{n+m-1}^{(1)}(LM)
      \ar[r]^-{\PD^{-1}}_-{\cong}
      &
      HH^{-n+1}_{(1)} (\wedge V)
    }
  \]
is commutative, where $\partial$ is the adjoint map.
\end{thm}



We note that Theorem \ref{thm:Model2_aut} gives another proof of \cite[Theorem 2]{FT04}.
In order to prove Theorem \ref{thm:Model2_aut}, we first observe rational models for $\Omega \aut M_0$ and the adjoint map $\partial$ by using the rational models for function spaces due to Brown and Szczarba \cite{BS97}.
Remark that the proof of the theorem due to F\'elix and Thomas uses a rational model for $\Gamma_1$ constructed by a  Haefliger model  \cite{Haefliger82} for the space of sections of a fibration.

Let $(\wedge (V\otimes A_*), d)$ and $\wedge S_\varphi = \left( \wedge \left( \overline{(V\otimes A_*)}^1 \oplus (V\otimes A_*)^{\geq 2} \right) , d \right)$ be the Brown-Szczarba models for $\Map (M, M)$ and $\aut M$, respectively. For the details of  Brown-Szczarba models, see \cite{BS97, BM06, HKO08} and also Appendix \ref{app:ModelAd}.

Let ${\mathcal M}_{S^1}=(\wedge (u), 0)$ be the Sullivan model for $S^1$ with $|u|=1$.
Since $\aut M$ is connected, nilpotent space \cite{HMR75} of finite type, it follows that the function space
 $L \aut M$ admits a Brown-Szczarba model of the form $(\wedge (S_\varphi \otimes \wedge (u)_*) ,d)$.
A Sullivan representative for the constant loop in $L \aut M$ at $id \in \aut M$ is of the form $\wedge S_\varphi \to \wedge (u)$ defined by $w\mapsto 0$ for $w\in S_\varphi$. This induces an augmentation  $\varepsilon : \wedge (S_\varphi \otimes \wedge (u)_*) \to \Q$
of the model for $L \aut M$.
Therefore, by virtue of \cite[Corollary 4.7]{BM06}, we have
a model
\[
(\wedge S_{\varepsilon} , d) = \left( \wedge \left( \overline{(S_\varphi \otimes \wedge (u)_*)}^1 \oplus (S_\varphi \otimes \wedge (u)_*)^{\geq 2} \right) , d \right),
\]
for the connected component $L \aut M_0$ of $L \aut M$ containing the constant loop. Moreover, it follows from
 \cite[Proposition 4,2, Theorem 4.5]{BM06} that the CDGA morphism
\begin{equation}
\wedge (S_\varphi \otimes \wedge (u)_*) \to \wedge S_{\varepsilon},
\hspace{1em}
w\otimes \beta \mapsto
\left\{
\begin{array}{ll}
\pr (w\otimes \beta) & (  |w\otimes \beta| \geq 1 )\\
0 & ( |w\otimes \beta|=0)
\end{array}
\right.
\end{equation}
for $w\in S_\varphi$ and $\beta \in \wedge (u)_*$ is a model for the inclusion $L \aut M_0 \hookrightarrow L\aut M$, where $\pr$ is the projection $(S_\varphi \otimes \wedge (u)_*)^{\geq 1} \to S_{\varepsilon}$.
The result \cite[Corollary 4.7]{BM06} yields that a morphism
\[
\omega_0 : \wedge S_\varphi \to \wedge S_{\varepsilon}
\]
of CDGAs, which is  defined by the projection onto $\overline{(S_\varphi \otimes \wedge (u)_*)}^1$ in $S^1_\varphi$ and $\omega_0 (w)=w\otimes 1^*$ for  $w \in S^{\geq 2}_\varphi$, is a rational model of the evaluation map at the base point $\ev_0 : L \aut M_0 \to \aut M$.
\begin{lem}\label{lem:model_OmegaAut1}
The fiber of $\omega_0$ at the canonical augmentation of $\wedge S_\varphi$ over $\Q$ is a Sullivan model for $\Omega \aut M_0$.
\end{lem}

We remark that the result \cite[Corollary 4.8]{BM06} is not applicable to the morphism $\omega_0$ since $\aut M$ is not a simply-connected in general.

\begin{proof}[Proof of Lemma \ref{lem:model_OmegaAut1}]
For proving the assertion, it is enough to show that $\omega_0$ is a KS-extension; that is, $\wedge S_{\varepsilon}$ is a relative Sullivan algebra with base $\wedge S_\varphi$ and $\omega_0$ is the canonical inclusion; see, for example, \cite[Section 14]{FHT1} for relative Sullivan algebras.
Observe that
\[
(S_\varphi \otimes \wedge (u)_* )^0 = S_\varphi^1 \otimes \Q u^*
\hspace{1em}
\text{and}
\hspace{1em}
(S_\varphi \otimes \wedge (u)_* )^1 = (S_\varphi^1 \otimes \Q 1^*) \oplus (S_\varphi^2 \otimes \Q u^*).
\]
It is readily seen that $\overline{(S_\varphi \otimes \wedge (u)_*)}^1$ coincides with the complement of the morphism
\[
\xymatrix@C40pt{
S_\varphi^1 \otimes  \Q u^*
\ar[r]^-{(0, d_0\otimes 1)}
&
(S_\varphi^1 \otimes  \Q 1^*) \oplus (S_\varphi^2 \otimes  \Q u^*),
}
\]
where $d_0$ is the linear part of the differential of $\wedge S_\varphi$.
It follows that $\wedge S_{\varepsilon}$ is isomorphic to
\[
\wedge \left( (S_\varphi \otimes \Q 1^*) \oplus \left( \overline{S^2_\varphi} \oplus S_\varphi^{\geq 3}    \right) \otimes \Q u^*   \right)
\cong
\wedge S_\varphi \otimes \wedge \left( \left( \overline{S^2_\varphi} \oplus S_\varphi^{\geq 3}    \right) \otimes \Q u^*   \right)
\]
which is a relative Sullivan algebra with base $\wedge S_\varphi$.
Here, $\overline{S^2_\varphi}$ is the quotient space $S_\varphi^2/d_0 (S_\varphi^1)$.
Therefore, the morphism $\omega _0$ of CDGAs is a KS-extension with the fiber $\wedge \left( \left( \overline{S^2_\varphi} \oplus S_\varphi^{\geq 3}    \right) \otimes \Q u^*   \right)$.
Since $\aut M$ is an H-space,  we have a homotopy equivalence
\[
\aut M \times \Omega \aut M_0  \simeq L \aut M_0
\]
defined by
$(x, \gamma) \mapsto x\cdot \gamma$. A homotopy, which defines
the holonomy action of $\pi_1 (\aut M)$ on $H^* (\Omega \aut M_0)$, factors through the product. This implies that the $\pi_1 (\aut M)$-action
is trivial and hence nilpotent.
Therefore, by virtue of \cite[20.3 Theorem]{Halperin83}, we see that the fiber $\wedge \left( \left( \overline{S^2_\varphi} \oplus S_\varphi^{\geq 3}    \right) \otimes \Q u^* \right)$ is a Sullivan model for $\Omega\aut M_0$.
\end{proof}

Now, we recall facts on rational homotopy groups of nilpotent spaces.
Let $X$ be a connected nilpotent space of finite type and ${\mathcal M}_X=(\wedge W,d)$ a Sullivan model for $X$. It follows from \cite[11.3]{BG} that there is a natural isomorphism
\[
  \nu : \pi_n (X) \otimes \Q \longrightarrow \Hom (H^n(W,d_0), \Q)
\]
provided $\pi_n (X)$ is abelian, where $d_0$ is the linear part of the differential $d$.
Let $f:S^n \to X$ be a map which represents an element in $\pi_n (X)\otimes \Q$.
Then, the image $\nu (f)$ is defined by the linear part of ${\mathcal M}_f$ a Sullivan representative of the map $f$.
We denote by ${\mathcal M}_{S^n}$ the Sullivan model for $S^n$ described in \cite[\S 12 Example 1]{FHT1} and $\int_{S^n} : {\mathcal M}_{S^n} \to \Q$ the chain map which assigns $1$ to a representative of the fundamental class of $S^n$.
Since ${\mathcal M}_f (w)$ is indecomposable for any $w$ in $(\wedge W)^n$,
it follows that $\nu(f)$ coincides with the map induced by the chain map $\int_{S^n} \circ {\mathcal M}_f |_{W}$ on $H^n(W, d_0)$.

Let $M$ be a simply-connected manifold. The monoid $\aut M$ and $\Omega\aut M_0$ are connected nilpotent H-spaces; see \cite{HMR75}.
Therefore, by the models mentioned above, we see that the dual spaces of $\pi_n (\aut M)\otimes \Q$ and $\pi_n (\Omega_* \aut M)\otimes \Q$ are isomorphic to the homology of $S_\varphi$ and  $ \left( \overline{S^2_\varphi} \oplus S_\varphi^{\geq 3} \right) \otimes \Q u^* $ the linear parts of the Sullivan models for $\aut M$ and $\Omega \aut M_0$, respectively. Observe that the fundamental groups $\pi_1 (\aut M)$ and $\pi_1 (\Omega \aut M_0)$ are abelian.

In what follows, we put
\[
T_\varphi := \overline{S^2_\varphi} \oplus S_\varphi^{\geq 3},
\]
where the differential of $T_\varphi$ is induced by the linear part $(S_\varphi , d_0)$.
Let $\iota : S_\varphi \to  T_\varphi \otimes \Q u^* $ be the composite of the inclusion $ T_\varphi \hookrightarrow T_\varphi \otimes \Q u^* $ defined by $w \mapsto (-1)^{|w|}w\otimes u^*$ and the projection $ \pr' : (S_\varphi)^{\geq 2} \to T_\varphi$.

\begin{prop}\label{prop:ModelConnHom}
The morphism $\iota$ is a rational model for the adjoint map $\partial$; that is, the diagram
\begin{equation}\label{diag:ModelConnHom}
    \xymatrix@C30pt@R20pt{
      \pi_{n-1} (\Omega \aut M_0) \otimes \Q
      \ar[d]^-{\partial}_-{\cong}
      \ar[r]^-{\nu}_-{\cong}
      &
      \Hom_{\Q}\left( H^{n-1}\left( T_\varphi \otimes \Q u^* \right), \Q \right)
      \ar[d]^-{H(\iota)^*}
      \\
      \pi_{n}(\aut M)\otimes \Q
      \ar[r]^-{\nu}_-{\cong}
      &
      \Hom_{\Q}(H^{n}(S_\varphi) , \Q ),
    }
  \end{equation}
  is commutative for $n\geq 2$.
  As a consequence, the morphism
$H(\iota)  : H^n (S_\varphi )   \longrightarrow H^{n-1}( T_\varphi \otimes \Q u^*)
$
induced by $\iota$ is an isomorphism for $n\geq 2$.

\end{prop}

\begin{proof} We may assume that $M$ is a rational space without loss of generality.
  Let $f : S^{n-1} \to \Omega \aut M_0$ be a based map which represents an element in $\pi_{n-1} (\Omega \aut M_0)\cong
  \pi_{n-1} (\Omega \aut M_0) \otimes \Q$.
  Consider a commutative diagram
\begin{equation}\label{ConnHom}
    \xymatrix@C40pt@R20pt{
      S^n
      \ar@{=}[d]
      \ar@/^10pt/[rrd]^-{\partial (f)}
      &&&
      \\
      S^{n-1} \wedge S^1
      \ar[r]^-{f\wedge 1}
      &
      \Omega \aut M_0 \wedge S^1
      \ar[r]^-{\ev}
      &
      \aut M
      \\
      S^{n-1} \times S^1
      \ar[r]^-{f\times 1}
      \ar@{->>}[u]
      &
      \Omega \aut M_0 \times S^1
      \ar[r]^-{\ev}
      \ar@{->>}[u]
      &
      \aut M,
      \ar@{=}[u]
    }
  \end{equation}
where $\ev$ is the evaluation map.
Let $\inc' : \Omega \aut M_0 \hookrightarrow L \aut M$ be the inclusion.
By the construction of the model for $\Omega \aut M_0$ in the proof of Lemma \ref{lem:model_OmegaAut1}, the CDGA morphism ${\mathcal M}_{\inc'} : \wedge (S_\varphi \otimes \wedge (u)_*) \to \wedge \left( T_\varphi \otimes \Q u^*   \right)$ defined by ${\mathcal M}_{\inc'} (w\otimes 1^*) = 0$ and
\[
{\mathcal M}_{\inc'}(w\otimes u^*)
=
\left\{
\begin{array}{ll}
0 & (|w|=1)\\
\pr' (w) \otimes u^* & (|w|\geq 2)
\end{array}
\right.
\]
for $w\in S_\varphi$ is a rational model for $\inc'$.
Therefore, by combining the rational models of $\inc'$ and the evaluation map $L\aut M \times S^1 \to \aut M$ due to Buijs and Murillo \cite{BM06}, we see that $\ev$ admits a Sullivan representative
\[
{\mathcal M}_{\ev}: \wedge S_\varphi \to \wedge \left( T_\varphi \otimes \Q u^*   \right) \otimes \wedge (u)
\]
defined by
${\mathcal M}_{\ev}(w) = - ( \pr'(w)\otimes u^* ) \otimes u$.
  Let ${\mathcal M}_f : \wedge (T_\varphi \otimes \Q u^*) \to {\mathcal M}_{S^{n-1}}$ and ${\mathcal M}_{\partial (f)} : \wedge S_\varphi \to {\mathcal M}_{S^n}$ be Sullivan representatives of $f$ and $\partial (f)$, respectively.
 Then we have the following homotopy commutative diagram of CDGAs
\[
\xymatrix@C50pt@R15pt{
&
&
{\mathcal M}_{S^n}
\ar[r]^-{\int_{S^n}}
 \ar[d]^-{\pi}
&
\Q
\ar@{=}[d]
\\
\wedge S_\varphi
\ar[r]^-{{\mathcal M}_{\ev}}
\ar@/^15pt/[rru]^-{{\mathcal M}_{\partial(f)}}
\ar[rd]_-{\wedge \iota}
&
\wedge (T_\varphi \otimes \Q u^*) \otimes \wedge (u)
\ar[r]^-{{\mathcal M}_f \otimes 1}
\ar[d]^-{1\otimes \int_{S^1}}
&
{\mathcal M}_{S^{n-1}} \otimes \wedge (u)
\ar[r]^-{\int_{S^{n-1}}\otimes \int_{S^1}}
\ar[d]^-{1\otimes \int_{S^1}}
&
\Q
\ar@{=}[d]
\\
&
\wedge (T_\varphi \otimes \Q u^*)
\ar[r]^-{{\mathcal M}_f}
&
{\mathcal M}_{S^{n-1}}
\ar[r]^-{\int_{S^{n-1}}}
&
\Q,
}
\]
  where $\pi$ is the canonical morphism which sends the fundamental class of $S^n$ to the fundamental class of $S^{n-1} \times S^1$ on cohomology.
The uniqueness up to homotopy of a Sullivan representative and commutativity of the diagram (\ref{ConnHom}) enable us to conclude that the top and left-hand side diagram is homotopy commutative.
  Therefore, we have
  \begin{align}
    \nu \circ  \partial (f)
    &=
    \left( \textstyle\int_{S^n} \circ {{\mathcal M}_{\partial (f)}} \right)|_{S_\varphi}
    \simeq
    \left( \textstyle\int_{S^{n-1}} \circ {\mathcal M}_f \circ { \wedge \iota} \right)|_{S_\varphi}
    \\
    &=
    \left( \textstyle\int_{S^{n-1}} \circ {{\mathcal M}_f} \right) |_{T_\varphi \otimes \Q u^*} \circ \iota
    =
    \nu (f) \circ \iota.
  \end{align}
  This completes the proof.
\end{proof}


We next consider a rational model for $\Gamma_1$.
Let $\wedge (V\otimes \wedge (u)_*)$ be the Brown--Szczarba model for $LM$, where $\wedge V$ is a Sullivan model for $M$.
Then we identify  the model with the CDGA $\HochschildModel$ mentioned in Section \ref{sect:preliminaries} by the isomorphism
$
\xi :
\HochschildModel \stackrel{\cong}{\longrightarrow} \wedge (V\otimes \wedge (u)_*),
$
defined by
$v \mapsto v\otimes 1^*$ and
$
\bar{v} \mapsto (-1)^{|v|} v\otimes u^*
$
for $v\in V$.
\begin{lem}\label{lem:ModelLf}
Let $\wedge V_i$ be a minimal Sullivan model of a simply-connected space $X_i$ for each $i=1$ and $2$. Let $\HochschildModel_{X_i}$ be the Sullivan model for $LX_i$ and $\psi : \wedge V_2 \to \wedge V_1 $ a Sullivan representative for $f : X_1 \to X_2$. Then, a CDGA morphism $\HochschildModel \psi : \HochschildModel_{X_2} \to \HochschildModel_{X_1}$ defined by
$\HochschildModel \psi (v) = \varphi (v)$ and  $\HochschildModel \psi (\bar{v}) = s\psi (v)$
for $v\in V_2$ is a Sullivan representative for the map $Lf : LX_1 \to LX_2$.
\end{lem}

\begin{proof}
Let $\wedge (V_i \otimes \wedge (u)_*) \cong \wedge (\wedge V_i \otimes \wedge (u_*))/{\mathcal I}$ be the Brown--Szczarba model for $LX_i$; see Appendix \ref{app:ModelAd}.
A naturality of Brawn-Szczarba models implies that the induced morphism
\[
\wedge (\psi \otimes 1) : \wedge (\wedge V_2 \otimes \wedge (u_*))/{\mathcal I} \longrightarrow \wedge (\wedge V_2 \otimes \wedge (u_*))/{\mathcal I}
\]
is a rational model for $Lf$.
It is readily seen that the square
\[
\xymatrix{
\wedge (V_2 \otimes \wedge (u)_*)
\ar[d]_-{\rho}^-{\cong}
&
\HochschildModel_{X_2}
\ar[l]_-{\xi}^-{\cong}
\ar[r]^-{\HochschildModel \varphi}
&
\HochschildModel_{X_1}
\ar[r]^-{\xi}_-{\cong}
&
\wedge (V_1 \otimes \wedge (u)_*)
\ar[d]^-{\rho}_-{\cong}
\\
\wedge (\wedge V_2 \otimes \wedge (u_*))/{\mathcal I}
\ar[rrr]^-{\wedge (\psi \otimes 1)}
&&&
\wedge (\wedge V_1 \otimes \wedge (u_*))/{\mathcal I}
}
\]
is commutative, which proves the lemma.
\end{proof}
By the restriction of the isomorphism $\xi$ mentioned above to the direct summands of the Hodge decomposition, we see that $\HochschildModel_{(k)}$ is isomorphic to $\wedge V \otimes \wedge ^k (V\otimes \Q u^*)$ which is a direct summand of $\wedge (V\otimes \wedge (u)_*)$.
Define the morphism ${\mathcal M}_{\Gamma_1}$ of CDGAs by the composite
\begin{eqnarray}\label{eq:Gamma_1}
\xymatrix@C20pt@R20pt{
\ \ \ \ \ \ \HochschildModel
\ar[r]^-{\text{proj}}
&
\HochschildModel_{(1)}
\ar[r]_-{\cong}^-{\xi}
&
 \wedge V \otimes (V\otimes \Q u^*)
\ar[r]^-{\varphi \otimes 1}
&
A\otimes (V\otimes \Q u^*)
\ar[lld]_-{D_A \otimes 1}^-{\cong}
\\
&
A_* \otimes (V\otimes \Q u^*)
\ar[r]_-{\cong}^-{T}
&
V\otimes A_* \otimes \Q u^*
\ar[r]^-{\pr'' \otimes 1}
&
T_\varphi \otimes \Q u^* .
}
\end{eqnarray}
Here, the map $T$ is defined by $T(a\otimes v\otimes u^*) = (-1)^{|a||v|}v\otimes a \otimes u^*$ and  $\pr'' : V \otimes A_* \to T_\varphi$ denotes the canonical projection.
\begin{prop}\label{prop:ModelGamma1}
The morphism ${\mathcal M}_{\Gamma_1}$ is a rational model for the dual of $\Gamma_1$.
\end{prop}
\begin{proof}
We first consider the composite
\[
\xymatrix{
g' :  \Omega \aut M_0
\ar[r]^-{\inc'}
&
L \aut M
\ar[r]^-{\ad}
&
\Map(S^1 \times M , M)
&
\Map(M , LM),
\ar[l]_-{\ad'}^-{\cong}
}
\]
where $\ad$ and $\ad'$ are the adjoint maps.
By virtue of Lemma \ref{model:ad}, we see that the map $\ad$ is modeled by the morphism ${\mathcal M}_{\ad} : \wedge (V \otimes (\wedge (u)\otimes A)_*) \to \wedge (S_\varphi \otimes \wedge (u)_*)$ given by
\[
{\mathcal M}_{\ad} (v \otimes \alpha )
  =
\rho^{-1}\left(
{\mathcal M}_{\inc}(v\otimes \alpha_0)\otimes 1^* + (-1)^{|\alpha_1|}{\mathcal M}_{\inc}(v\otimes \alpha_1)\otimes u^*
\right)
\]
for $\alpha \in (\wedge (u) \otimes A)_*$ and $\zeta (\alpha) = 1^* \otimes \alpha_0 + u^* \otimes \alpha_1$.
Here, $\rho$ and $\zeta$ are the isomorphisms described in Appendix \ref{app:ModelAd}.

Since $M$ is simply-connected, it follows that $LM$ is connected.
Thus, the same argument as in the proof of Lemma \ref{model:ad} enables us to obtain a model for
$\ad'$ of the form
$
{\mathcal M}_{\ad'} : \wedge (V\otimes (\wedge (u) \otimes A)_*)   \to \wedge ((V\otimes \wedge (u)_*)\otimes A_*)
$
which is induced by the isomorphisms $\zeta : (\wedge (u) \otimes A)_* \cong \wedge (u)_* \otimes A_*$.
Therefore, the composite ${\mathcal M}_{\inc'}\circ {\mathcal M}_{\ad} \circ {\mathcal M}_{\ad'}^{-1}$
is a model for $g'$, where ${\mathcal M}_{\inc'}$ is the model of $\inc'$ described in the proof of Proposition \ref{prop:ModelConnHom}.
Let $\ev' : \Map (M, LM) \times M \to LM$ be the evaluation map. Then we have a model for $\ev'$ of the form  ${\mathcal M}_{\ev'} : \wedge (V\otimes \wedge (u)_*) \to \wedge ((V\otimes \wedge (u)_*)\otimes A_*) \otimes A$ defined by
\[
{\mathcal M}_{\ev'}(v\otimes \beta) = \sum (-1)^{|a_i|} ((v\otimes \beta)\otimes a_i^*) \otimes a_i.
\]
This follows from \cite[Theorem 1.1]{BM06} and \cite[Theorem 4.5]{Kuribayashi06}.
We remark that the sign of the model ${\mathcal M}_{\ev'}$ is different from the original model due to Buijs and Murillo.
For details of the sign, see Remark \ref{rem:Sign_ModelEv} after the proof.

Since the map $g$ coincides with the composite $\ev' \circ (g' \times 1)$, it follows that
\[
{\mathcal M}_g = ({\mathcal M}_{\inc'}\circ {\mathcal M}_{\ad} \circ {\mathcal M}_{\ad'}^{-1} \otimes 1) \circ {\mathcal M}_{\ev'}
\]
is a rational model for $g$.
Explicitly, we compute
\begin{align}
{\mathcal M}_g (v)
&= ({\mathcal M}_{\inc'}\circ {\mathcal M}_{\ad} \circ {\mathcal M}_{\ad'}^{-1} \otimes 1) \circ {\mathcal M}_{\ev'} (v\otimes 1^*)
\\
&=({\mathcal M}_{\inc'}\circ {\mathcal M}_{\ad} \otimes 1) \left( \sum_i (-1)^{|a_i|} (v\otimes (1 \otimes a_i )^*) \otimes a_i \right)
\\
&=({\mathcal M}_{\inc'} \otimes 1) \left( \sum_i (-1)^{|a_i|}   \rho^{-1}\left( {\mathcal M}_{\inc}(v\otimes a_i^*) \otimes 1^* \right) \otimes a_i \right)
\\
&= \sum_{|a_i|=|v|} a_i^* (\varphi(v))\otimes a_i \ \ \ \text{and}
\\
{\mathcal M}_g (\bar{v})
&= (-1)^{|v|} ({\mathcal M}_{\inc'}\circ {\mathcal M}_{\ad} \circ {\mathcal M}_{\ad'}^{-1} \otimes 1) \circ {\mathcal M}_{\ev'} (v\otimes u^*)
\\
&= (-1)^{|v|} ({\mathcal M}_{\inc'}\circ {\mathcal M}_{\ad} \otimes 1) \left( \sum_i (-1)^{|a_i|} ((v\otimes (u \otimes a_i )^*) \otimes a_i \right)
\\
&=(-1)^{|v|}({\mathcal M}_{\inc'} \otimes 1) \left( \sum_i    \rho^{-1}\left( {\mathcal M}_{\inc}(v\otimes a_i^*) \otimes u^* \right) \otimes a_i \right)
\\
&=(-1)^{|v|}  \sum_{|a_i| < |v|}   \left( \pr'' ( v\otimes a_i^*) \otimes u^* \right) \otimes a_i.
\end{align}

Recall $\omega_A \in A^m$ the representative of the fundamental class $[M]$ described above, and define $\int_{\omega_A} : A \to \Q$ a linear map of degree $-m$ which maps $\omega_A$ to $1$. It is immediate that $\int_{\omega_A}$ is a rational model of the dual of $\Q \to H_*(M;\Q)$ defined by $1\mapsto [M]$.
Therefore, the definition of $\Gamma_1$ implies that the composite 
\begin{equation}\label{model:Gamma1_2}
\xymatrix@C20pt{
\HochschildModel
\ar[r]_-{\cong}^-{\xi}
&
\wedge (V\otimes \wedge (u)_*)
\ar[r]^-{{\mathcal M}_g}
&
\wedge \left( T_\varphi \otimes \Q u^*   \right) \otimes A
\ar[r]^-{1\otimes \int_{\omega_A}}
&
\wedge \left( T_\varphi \otimes \Q u^*   \right)
\ar[r]^-{\text{proj}}
&
T_\varphi \otimes \Q u^*
}
\end{equation}
induces the dual of $\Gamma_1$ on homology.

In order to complete the proof, it suffices to show that  the composite above coincides with ${\mathcal M}_{\Gamma_1}$ in (\ref{eq:Gamma_1}) on
$\HochschildModel_{(1)}$. We observe that $\varphi(v)= \sum_i a_i^*(\varphi(v))a_i$ for $v \in V$.
Moreover, we may write $D_A(a_{i_1}\cdots a_{i_k})= \sum_j \lambda_{(i_1, \ldots , i_k , j)}a_j^*$ for some $\lambda_{(i_1, \ldots , i_k , j)} \in \Q$. Thus it follows that
$
\int_{\omega_A}(a_{i_1}\cdots a_{i_k}a_i)=\omega_A^*(a_{i_1}\cdots a_{i_k}a_i) = D_A(a_{i_1}\cdots a_{i_k})(a_i) = \lambda_{(i_1, \ldots , i_k , i)}.
$
Then, the definition of  ${\mathcal M}_{\Gamma_1}$ yields that
\begin{align}
& {\mathcal M}_{\Gamma_1}(v_1 \cdots v_k \bar{v}) \\
= & ( \pr'' \otimes 1) \circ T\circ (D_A\otimes 1)\\
&\hspace{8em} \Big((-1)^{|v|}\sum_{(i_1, \ldots , i_k)}  a_{i_1}^*(\varphi(v_1)) \cdots a_{i_k}^*(\varphi(v_k))  a_{i_1}\cdots a_{i_k}\otimes v\otimes u^*\Big)\\
= & ( \pr'' \otimes 1) \circ T \Big((-1)^{|v|}\sum_{(i_1, \ldots , i_k)}  a_{i_1}^*(\varphi(v_1)) \cdots a_{i_k}^*(\varphi(v_k)) (\sum_i \lambda_{(i_1, \ldots , i_k , i)}a_i^*) \otimes v\otimes u^*\Big)\\
=& (-1)^{|v|}  \sum_{(i_1, \ldots , i_k ,i)} (-1)^{|a_i||v|} a_{i_1}^*(\varphi(v_1)) \cdots a_{i_k}^*(\varphi(v_k)) \pr'' (v\otimes a_i^*)\otimes u^*) \lambda_{(i_1, \ldots , i_k , i)}.
\end{align}
Moreover, we see that
\begin{align}
& (\text{proj}) \circ \left( 1\otimes \textstyle\int_{\omega_A}  \right) \circ {\mathcal M}_g \circ \xi (v_1 \cdots v_k \bar{v})\\
=& (-1)^{|v|}(\text{proj}) \circ \left( 1\otimes \textstyle\int_{\omega_A}  \right) \circ {\mathcal M}_g \left( (v_1 \otimes 1^*) \cdots (v_k \otimes 1^*) (v\otimes u^*)  \right)
\\
=&(-1)^{|v|}(\text{proj}) \circ \left( 1\otimes \textstyle\int_{\omega_A}  \right) \\
 & \left\{ \left(\sum_{(i_1, \ldots , i_k)}  a_{i_1}^*(\varphi(v_1)) \cdots a_{i_k}^*(\varphi(v_k))  a_{i_1}\cdots a_{i_k}
 \right)
 \left( \sum_i (\pr'' (v\otimes a_i^*)\otimes u^*) \otimes a_i \right) \right\}
\\
= & (-1)^{|v|}  \sum_{(i_1, \ldots , i_k ,i)} (-1)^{(|a_i|-m)(|v|+|a_i|+1)+ m(|v|+|a_i|+1)} \\
&\hspace{10em} a_{i_1}^*(\varphi(v_1)) \cdots a_{i_k}^*(\varphi(v_k)) \pr'' (v\otimes a_i^*)\otimes u^*) \lambda_{(i_1, \ldots , i_k , i)}.
\end{align}
Observe that $\lambda_{(i_1, \ldots , i_k , i)} = 0$ if $m \neq |a_{i_1}\cdots a_{i_k}| + |a_i|$.
Thus, we have the result.
\end{proof}

\begin{rem}\label{rem:Sign_ModelEv}
The sign of the rational model ${\mathcal M}_{\ev'}$ for the evaluation map in the proof of Proposition \ref{prop:ModelGamma1} is different from the model due to Buijs-Murillo \cite{BM06} and Kuribayashi \cite{Kuribayashi06}.
It is caused by the difference of sings appeared in the differential of the dual space $A_*$.
The differential $d_*$ of $A_*$ in\cite{BS97,
BM06, Kuribayashi06} is defined by $d_*(\alpha) = \alpha \circ d_A$ for $\alpha \in A_*$ with the differential $d_A$ of $A$.
In this paper, we adopt the differential $d_*$ of $A_*$ defined by $d_*(\alpha) = -(-1)^{|\alpha|} \alpha \circ d_A$ with the Koszul sign convention.
\end{rem}

In \cite[\S 2]{BM08}, Buijs and Murillo define a  quasi-isomorphism
\[
\Psi : \Der (\SpaceModel , A) \to \Hom (S_\varphi, \Q )
\]
by
$
\Psi (\theta)(v\otimes \alpha) := (-1)^{(|\theta | + |v|)|\alpha|}\alpha \circ \theta (v)
$
for $\theta \in \Der (\wedge V , A)$ and $v\otimes \alpha \in S_\varphi$.
Then, the isomorphism $H(\Psi)$ on homology is related to Sullivan's isomorphism $\Phi$.

\begin{lem}\label{lem:PhiTheta} There exists a commutative diagram
  \begin{equation}\label{diag:Model1_aut}
    \xymatrix@C20pt@R20pt{
      \pi_* (\aut M) \otimes \Q
      \ar[rr]^-{\Phi}_-{\cong}
      \ar[d]_-{\nu}^-{\cong}
      &&
      H^{-*} (\Der (\wedge V))
      \ar[d]^-{\varphi_*}_-{\cong}
      \\
      \Hom (  H^*(S_\varphi ), \Q)
      &
      H^{-*}( \Hom (S_\varphi , \Q ) )
      \ar[l]_-{\cong}
      &
      H^{-*} (\Der (\wedge V, A)),
      \ar[l]_-{H(\Psi)}^-{\cong}
    }
  \end{equation}
  where the unnamed arrow denotes the natural isomorphism.
\end{lem}

\begin{proof}
  Given $f \in \pi_n (\aut M) \otimes \Q$ which is represented by $f:S^n \to \aut M$. Here we assume that $M$ is a rational space.
  Let $ \ad (f) : S^n \times M \to M$ be the adjoint of $f$ and ${\mathcal M}_{\ad(f)}:\wedge V \to {\mathcal M}_{S^{n}}\otimes \wedge V$ a Sullivan representative for $\ad (f)$.
  Note that $(1\otimes \varphi) \circ {\mathcal M}_{\ad (f)}$ is also a Sullivan representative for $\ad (f)$.
  By the definition of $\Phi$, we have
  \begin{equation}\label{1}
    \varphi_* \circ \Phi (f)
    = \varphi_* \left\{ \left( \textstyle\int_{S^n} \otimes 1 \right) \circ {\mathcal M}_{\ad(f)} \right\}
    =\left( \textstyle\int_{S^n} \otimes 1 \right) \circ (1\otimes \varphi) \circ {\mathcal M}_{\ad(f)}
    .
  \end{equation}

  On the other hand, the adjoint $\ad (f)$ coincides with the composite
  \begin{equation}\label{ad(f)}
    \xymatrix{
      \ad (f) : S^n \times M
      \ar[r]^-{f\times 1}
      &
      \aut M \times M
      \ar[r]^-{\inc \times 1}
      &
      {\rm Map}(M,M)\times M
      \ar[r]^-{\ev''}
      &
      M,
    }
  \end{equation}
  where $\ev''$ is the evaluation map.
  Let ${\mathcal M}_{f} : \wedge S_\varphi \to {\mathcal M}_{S^n}$ be a Sullivan representative for $f$ and ${\mathcal M}_{\ev''} : \wedge V \to \wedge (V\otimes A_*) \otimes A$ the rational model for ${\ev''}$ defined by
\[
  {\mathcal M}_{\ev''}(v) = \sum_j (-1)^{|a_j|} (v\otimes a_j^*) \otimes a_j
\]
for $v\in V$.
Let ${\mathcal M}_{\inc}$ be the rational model for $\inc$ described in Appendix \ref{app:ModelAd}.
It follows from \eqref{ad(f)} that the composite  $({\mathcal M}_f\circ {\mathcal M}_{\inc} \otimes 1) \circ {\mathcal M}_{\ev''}$ of morphisms of CDGAs is also a rational model for $\ad (f)$ and then it is homotopic to $(1\otimes \varphi)\circ {\mathcal M}_{\ad(f)}$.
  Therefore, by \eqref{1}, we have
  \begin{align}
    \Psi \circ \varphi_* \circ \Phi (f) (v\otimes a_i^*)
    &=(-1)^{(|f|+|v|)|a_i|}a_i^* \left\{
      \left( (\textstyle\int_{S^n}  \circ{\mathcal M}_f\circ {\mathcal M}_{\inc} \otimes 1 \right) \circ {\mathcal M}_{\ev''}(v)\right\} \\
    &= (-1)^{(|f|+|v|)|a_i|}a_i^* \left\{  \sum_j (-1)^{|a_j|} \left( \textstyle\int_{S^n}\circ {\mathcal M}_f (v\otimes a_j^*) \right) a_j \right\} \\
    &= \textstyle\int_{S^n}\circ {\mathcal M}_f (v\otimes a_i^*) = \nu (f)(v\otimes a_i^*)
  \end{align}
  for $v\otimes a_i^* \in S_\varphi$.
\end{proof}

\begin{proof}[Proof of Theorem \ref{thm:Model2_aut}]
By making use of isomorphisms $\overline{V}\cong V\otimes \Q u^*$ and $S_\varphi \otimes \Q u^* \cong S_\varphi $ defined by $\bar{v} \mapsto v\otimes u^*$ and $w\otimes u^* \mapsto (-1)^{|w|}w$, respectively, we obtain morphisms $\lambda ' : \Der (\wedge V) \to \Hom (V\otimes \Q u^* , A)$ and $\Psi' : \Hom (V\otimes \Q u^* , A) \to \Hom (S_\varphi \otimes \Q u^* , \Q)$ of chain complexes which fit in the commutative diagram
\[
\xymatrix@C20pt@R20pt{
\Hom (\overline{V}, A)
&
\Der (\wedge V, A)
\ar[r]^-{\Psi}
\ar[d]^-{\lambda'}_-{\cong}
\ar[l]_-{\lambda}^-{\cong}
&
\Hom (S_\varphi, \Q)
\ar[d]^-{\cong}
\\
&
\Hom (V\otimes \Q u^*, A)
\ar[r]^-{\Psi'}
\ar[lu]^-{\cong}
&
\Hom (S_\varphi \otimes \Q u^*, \Q).
}
\]
Recall the morphisms $\PD$ in (\ref{eq:PD}) and ${\mathcal M}_{\Gamma_1}$ in (\ref{eq:Gamma_1}).
Then, a straightforward computation shows that the following diagram
\[
\xymatrix@C20pt@R20pt{
\Hom (S_\varphi , \Q)
\ar[rd]^-{\cong}
&&
\Der (\wedge V , A)
\ar[ll]_-{\Psi}
\ar[ddd]^-{\lambda}
\ar@/^20pt/[ldd]_-{\lambda'}
\\
\Hom (T_{\varphi} \otimes  \Q u^*, \Q)
\ar@/_85pt/[dd]_-{{\mathcal M}_{\Gamma_1}^*}
\ar[d]_-{(\pr'' \otimes 1)^*}
\ar[u]^-{\iota}
&
\Hom (S_\varphi \otimes \Q u^* , \Q)
\ar[ld]_-{(\pr \otimes 1)^*}
\\
\Hom (A_* \otimes V \otimes \Q u^* , \Q )
\ar[d]_-{ \left( (D_A \circ \varphi \otimes 1) \circ \xi \right)^*}
&
\Hom (V \otimes \Q u^* , A)
\ar[u]^-{\Psi '}
\ar[l]_-{\text{adj}}^-{\cong}
\ar[rd]^-{\cong}
\\
\Hom (\HochschildModel_{(1)} , \Q)
&&
\Hom (\overline{V} , A).
\ar[ll]_-{{\rm PD}}
}
\]
is commutative.
Therefore, by Proposition \ref{prop:ModelConnHom}, \ref{prop:ModelGamma1} and Lemma  \ref{lem:PhiTheta}, we have the commutativity of the diagram in the assertion.
\end{proof}


\section{Examples}\label{sect:examples}
In this section, we describe explicitly the Lie representation  $L$ and the contraction $e$ in Propositions \ref{prop:LieAlg-New} and
\ref{prop:Lie_e} for manifolds and interesting spaces.
For a Sullivan algebra $\wedge V$, we denote by $(v,\alpha)$ the derivation on $\wedge V$ that takes a generator $v$ in $V$ to an element $\alpha$ in $\wedge V$ and the other generators to $0$.

\begin{ex} \label{ex:G} Let $G$ be a
simply-connected compact Lie group and $\alpha_1, \ldots, \alpha_l$  the indecomposable elements of $H^*(G;\Q)$ with maximal degree.
  The cohomology $H^*(LG;\Q)$ is generated by $\alpha_1, \ldots,\alpha_l$ as an algebra with the BV operator and derivations $L_\theta$ for suitable elements $ \theta \in
  \pi_*(\aut G)\otimes \Q$. In this case, the Lie derivative $L : \pi_*(\aut G)\otimes\Q \to \Der_*(H^*(LG; \Q))$ is faithful.
\end{ex}

\begin{ex}\label{ex:CP^2}
Let $X$ be the complex projective plane $\C P^2$.  Then a Sullivan minimal model ${\mathcal M}_X$ for $X$ is  given by  $(\wedge (x,y),d)$ with $|x|=2$, $|y|=5$, $dx=0$ and $dy=x^3$. Moreover, the free loop space $LX$ admits the Sullivan minimal model $\HochschildModel = (\wedge (x,y,\bar{x},\bar{y}),d)$ for which $|\bar{x}|=1$, $|\bar{y}|=4$, $dx=0$, $dy=x^{3}$, $d\bar{x}=0$ and $d\bar{y}=-3x^2\bar{x}$.
Recall the result \cite[Theorem 2.2(ii)]{KY97} which asserts that  $$H^*(LX;\Q )\cong \frac{\Q[x]\otimes \wedge (\bar{x})}{(x^3,x^2\bar{x})}\oplus \Bigl((x,\bar{x})_A\otimes \Q^+[z]\Bigr)\ ; \  \ \ \ |z|=|\bar{y}|=4$$
as an algebra, where $(x,\bar{x})_A$ is the ideal of $A:= {\Q[x]\otimes \wedge (\bar{x})}/{(x^3,x^2\bar{x})}$. We choose a basis for $H^*(LX;\Q )$ of the form
$$\{ 1,\ x\ ,x^2,\ \bar{x}\ (=\alpha_0) \ , x\bar{x}, \ \alpha_n, \ x\alpha_n, \ \beta_n, x\beta_n \}_{n\geq 1},$$
where
$\alpha_n=\bar{x}\bar{y}^n$ and $\beta_n= {x}\bar{y}^n+3n\bar{x}y\bar{y}^{n-1}$.
Observe that $$H_*(\Der ({\mathcal M}_X))=\Q \{ (y,1),(y,x)\}.$$
Let $e_1:=e_{(y,1)}=-(\bar{y},1)$,  $e_2:=e_{(y,x)}=(\bar{y}, {x})$,
$L_1:=L_{(y,1)}=(y,1)$ and  $L_2:=L_{(y,x)}=(y,x)+(\bar{y}, \bar{x})$.
Then we see that
$e_1(\alpha_n)=n\alpha_{n-1}$,  $e_2(\alpha_n)=nx\alpha_{n-1}$, $e_1(\beta_1)=-x$, $e_1(\beta_n)=-n\beta_{n-1}$ $(n>1)$,  $e_2(\beta_1)=x^2$, $e_2(\beta_n)=nx\beta_{n-1}$ $(n>1)$,
$L_1(\alpha_n)=L_2(\alpha_n)=0$, $L_1(\beta_n)=-3n\alpha_{n-1}$ and  $L_2(\beta_n)=-2nx\alpha_{n-1}$. Thus $L$ and $e$ are injective.

Note that the calculations of the operations $L$ and $e$ yield that
$e_1(x\otimes z^n)=-n{x}\otimes z^{n-1}$,
$e_2(x\otimes z^n)=nx^2\otimes z^{n-1}$,
$e_1(\bar{x}\otimes z^n)=n\bar{x}\otimes z^{n-1}$,
$e_2(\bar{x}\otimes z^n)=nx\bar{x}\otimes z^{n-1}$,
$L_1(x\otimes z^n)=-3n\bar{x}\otimes z^{n-1}$,
$L_2(x\otimes z^n)=-2nx^2\otimes z^{n-1}$ and
$L_1(\bar{x}\otimes z^n)=L_2(\bar{x}\otimes z^n)=0$.
\end{ex}

\begin{rem}\label{rem:CP^2}
We see that $H_*(\Der(H^*(\C P^2;\Q))) =\Der_*(H^*(\C P^2;\Q)) =0$. In fact, with the same notation as in
\cref{ex:CP^2}, every derivation assigning an element in $H^0(\C P^2;\Q)=\Q$ to the generator $x$ should be trivial. As mentioned above,
the homology Lie algebra $H_*(\Der({\mathcal M}_{\C P^2}))$ is non trivial. Observe that $\C P^2$ is {\it formal}; see \cite[\S 12(c)]{FHT1}.
Thus, a quasi-isomorphism does not induce
an isomorphism between the homology Lie algebras of derivations in general.
\end{rem}

\begin{ex}
Let $X$ be a non-formal space whose minimal model ${\mathcal M}_X$ has the form
$(\wedge (x,y,z),d)$, where $dx=dy=0$, $dz=xy$,
$|x|=|y|=3$ and $|z|=5$. We observe that ${\mathcal M}_X$
 is realized by a manifold of dimension $11$; see  \cite[Theorem 13.2]{S}.
Then $H_*(\Der ({\mathcal M}_X))=\Q \{ (z,1)\}$
and $H^*(X;\Q )=\wedge (x,y)\otimes \Q [w,u]/(xy, xw, yu, xu+yw, w^2,wu,u^2)$
for $w=[xz]$ and $u=[yz]$.
Thus the natural map $\psi :H_*(\Der ({\mathcal M}_X))\to \Der_* (H^*(X;\Q ))$ is faithful
since $\psi {(z,1)}(w)=[x]$.
Moreover
$e$ and $L$ are injective since
$e_{(z,1)}([xyz\bar{z}])=[xy{z}]$ and  $L_{(z,1)}([xyz\bar{z}])=[xy\bar{z}]$.
\end{ex}

In the following examples,
we rely on the software {\it Kohomology} \cite{Wakatsuki2}
for determining bases for $H^*(X; \Q)$ and $H^*(LX; \Q)$ and
computing actions of \(L_\theta\) and \(e_\theta\) on \(H^*(LX; \Q)\)
with data of 
a Sullivan model for a given space $X$.

\begin{ex}
Let $X$ be a non-formal manifold of dimension $14$ whose minimal model
${\mathcal M}_X$ is of the form $(\wedge (a,x,y,b,v,w),d)$, where
$|a|=2$, $|x|=|y|=3$, $|b|=4$, $|v|=5$, $|w|=7$,
$da=dx=0$, $dy=a^2$, $db=ax$, $dv=ab+xy$, $dw=2xv+b^2$; see \cite[p.439]{FHT1} and \cite[Theorem 13.2]{S}.
Note that $H^*(X;\Q )$ is generated by $$\{ a, x, xb, av-yb, a^2w-abv+xyv, 3axw+b^3
\}$$
as an algebra.
Then $H_*(\Der ({\mathcal M}_X))=\Q \{ (w,1), (w,a)\}$ and
the natural map $\psi :H_*(\Der ({\mathcal M}_X))\to \Der_* ( H^*(X;\Q ))$ is zero though $\Der_* (H^*(X;\Q ))\neq 0$; see \cite{Yam05}.
However, the Lie representation $L:H_*(\Der ({\mathcal M}_X))\to \Der_* (H^*(LX;\Q ))$
is non-trivial.
 Indeed, we see that $L_{(w,1)}\neq 0$ and $L_{(w,a)}\neq 0$ since
\begin{align}
    &L_{(w,1)}(xbv\bar{w}-axw\bar{w})=2xv\bar{b}+ax\bar{w}  \ \ \text{and}\\
    &L_{(w,a)}(-axw\bar{w} +  xbv\bar{w} + xvw\bar{b} )=3 (\frac{1}{2} a^2x\bar{w} + axv\bar{b}) +(axb\bar{v} + axw\bar{a} + xyb\bar{b})
\end{align}
as non-zero cohomology classes \cite{Wakatsuki2}.
Thus it follows that the representation
$L:H_*(\Der ({\mathcal M}_X))\to \End_* ( H^*({\HochschildModel}_{(1)} ))$
is faithful. Here ${\HochschildModel}_{(k)}=\wedge V\otimes \wedge^k \overline{V}$ for ${\mathcal M}_X=(\wedge V,d)$ in Section \ref{sect:preliminaries}.

The contraction $e$ is injective as seen in  Corollary \ref{cor:e_faithful}. In this case, we can check the faithfulness with explicit calculations. Indeed, we see in  \cite{Wakatsuki2} that
\begin{align}
    &e_{(w,1)}(a^2xw\bar{w}-axbv\bar{w}-2axvw\bar{b})=a^2xw-xbva \ \ \text{and}\\
    &e_{(w,a)}(axw\bar{w}-xbv\bar{w}-2xvw\bar{b})=a^2xw-xbva,
\end{align}
where $[a^2xw-xbva]\in H^{14}(X;\Q)(=H^{14}( {\HochschildModel}_{(0)}))$ is the fundamental class of $X$.
\end{ex}

\begin{ex}\label{ex:355}
When $X$ does not have positive weights, $L:H_*(\Der ({\mathcal M}_X))\to \End_* (H^*({\HochschildModel}_{(1)} ))$
may not be faithful.
Let $X$ be an elliptic manifold of dimension $228$ with  $${\mathcal M}_X=(\wedge (x_1, x_2, y_1, y_2, y_3, z), d)$$
given in \cite[Example 5.2]{AL}, where
\(\deg{x_1} = 10\), \(\deg{x_2} = 12\), \(\deg{y_1} = 41\), \(\deg{y_2} = 43\), \(\deg{y_3} = 45\), \(\deg{z} = 119\) and
the differential is defined by
\begin{align}
  dx_1 &= 0 & dy_1 &= x_1^3 x_2
  & dz &= x_2(y_1 x_2 - x_1 y_2)(y_2 x_2 - x_1 y_3) + x_1^{12} + x_2^{10}. \\
  dx_2 &= 0 & dy_2 &= x_1^2 x_2^2 \\
       & & dy_3 &= x_1 x_2^3
\end{align}
Then we see that $L_{\theta}:H_*(\Der ({\mathcal M}_X))\to \End_* (H^*({\HochschildModel}_{(1)} ))$ is not zero except  $\theta=(z,x_2^9)$, $(z,x_1^2x_2)$ and  $(z,x_1x_2^2)$ \cite{Wakatsuki2}.
Observe that
 $L_{\theta}:H_*(\Der ({\mathcal M}_X))\to \End_* (H^*({\HochschildModel}_{(4)} ))$
for $\theta=(z,x_1^2x_2)$ and $(z,x_1x_2^2)$ is non trivial for elements of degree $221$ and $219$ of ${H^*({\HochschildModel}_{(4)} )}$, respectively. Unfortunately, a calculation of $L_{(z,x_2^9)}:
H_*(\Der ({\mathcal M}_X))\to \Der_* (H^{*}(LX;\Q))$ with Kohomology \cite{Wakatsuki2} shows that the representation is trivial for degrees less than or equal to $355$.
\end{ex} 

We do not know whether the operator $L$ is a faithful representation in general. In \cref{ex:355}, it is expected that $L_\theta$ is not zero for some derivation $\theta$ with higher degree.

\begin{prob}
Is $L:H_*(\Der ({\mathcal M}_X))\to \Der_* (H^*(LX;\Q ))$ faithful when $X$ is a closed manifold ?
\end{prob}

\section*{Acknowledgements}
The authors thank the referees for their careful reading and valuable comments.

\appendix

\section{A Sullivan representative for an adjoint map}
\label{app:ModelAd}

We begin by recalling the rational models due to Brown and Szczarba \cite{BS97}.
Let $\wedge V$ be a minimal Sullivan algebra, $A$ a finite dimensional CDGA and $A^q = 0$ for $q<0$.
We denote by $A_* = \Hom (A,\Q)$ the dual of $A$ with the coproduct $\Delta_A$ of $A_*$ induced by the multiplication of $A$; see \cref{rem:Sign_ModelEv}.
We consider the CDGA $\wedge (\wedge V \otimes A_*)$ and the differential ideal ${\mathcal I}$ of $\wedge (\wedge V \otimes A_*)$ generated by $1\otimes 1^* - 1$ and
\[
w_1 w_2 \otimes \alpha - \sum (-1)^{|w_2||\alpha'_i|}(w_1\otimes \alpha'_i)(w_2\otimes \alpha''_i),
\]
where $w_i \in \wedge V$, $\alpha \in A_*$ and $\Delta_A (\alpha) = \sum \alpha'_i \otimes \alpha''_i$.
Then, it follows from \cite[Theorem 3.5]{BS97} that the composite
\[
\xymatrix{
\rho :
\wedge (V\otimes A_*)
\ar[r]^-{\text{incl}}
&
\wedge (\wedge V\otimes A_*)
\ar[r]^-{\text{proj}}
&
\wedge (\wedge V\otimes A_*)/{\mathcal I}
}
\]
is an isomorphism of graded algebras. We define the differential $d_{BS}$ of $\wedge (V\otimes A_*)$ by $\rho^{-1}d\rho$, where $d$ is the differential of $\wedge (\wedge V\otimes A_*)/{\mathcal I}$.

Assume that $\wedge V$ is a minimal Sullivan model for a connected nilpotent space $Y$ of finite type and $A$ is a finite dimensional commutative model for a finite CW complex $X$.
Then, we see that $(\wedge (V\otimes A_*), d_{BS})$ is a rational model of $\Map(X,Y)$; see \cite[Theorem 1.3]{BS97}.

Let  $\varphi : \wedge V \to A$ be a Sullivan representative for a continuous map $f:X\to Y$. The morphism of CDGAs induces the augmentation $\varphi : \wedge (V\otimes A_*) \to \Q$ which is denoted by the same notation.
It follows from \cite[Proposition 4.2, Theorem 4.5]{BM06} and \cite[Remark 3.4]{HKO08} that the connected component $\Map_f (X,Y)$ of $\Map(X,Y)$ containing $f$ has a Sullivan model of the form
\[
(\wedge S_\varphi , d) = \left( \wedge \left( \overline{(V\otimes A_*)}^1 \oplus (V\otimes A_*)^{\geq 2} \right) , d \right),
\]
where $\overline{(V\otimes A_*)}^1$ is the complement of the image of the composite
\[
\xymatrix{
(V\otimes A_*)^0
\ar[r]^-{d}
&
\left( \wedge (V\otimes A_*) \right)^1
\ar@{->>}[r]
&
\left( \wedge (V\otimes A_*) / K_{\varphi} \right)^1
&
(V\otimes A_*)^1
\ar[l]^-{\cong}_-{\text{proj}}
}
\]
in which $K_\varphi$ is the differential ideal of $\wedge (V\otimes A_*)$ generated by $(V\otimes A_*)^{<0}$ and $\{ w - \varphi(w) \mid w\in (V\otimes A_*)^0 \}$. We observe that the differential $d$ is induced by the differential $d_{BS}$ of $\wedge (V\otimes A_*)$; see \cite[The proof of Proposition 4.2]{BM06}.
Moreover, the morphism ${\mathcal M}_{\inc} : \wedge (V\otimes A_*) \to \wedge S_\varphi$ of CDGAs defined by
  \[
    {\mathcal M}_{\inc}(w)
    = \left\{
      \begin{array}{ll}
        \pr (w) & (|w| >  0), \\
        \varphi(w) & ( |w|=0 ),\\
        0 & (|w|<0)
      \end{array}
    \right.
  \]
is a rational model for the inclusion $\inc : \Map_f (X,Y) \hookrightarrow \Map (X,Y)$, where $\pr  : (V\otimes A_*)^{\geq 1} \to S_\varphi$ is the canonical projection.

In what follows, let $X_i$ and $Y_i$ be connected nilpotent spaces of finite type for $i=1,2$.
We further assume that $X_i$ is a finite CW complex.
Moreover, let $A_i$ be a finite dimensional commutative model for $X_i$ and $\wedge V_i$ a minimal Sullivan model for $Y_i$.
We first construct a rational model for the map
\[
- \times g : \Map (X_1 , Y_1) \to \Map (X_1 \times X_2 , Y_1 \times Y_2)
\]
defined by $f \mapsto f \times g$ with a continuous map $g : X_2 \to Y_2$.
Let $\iota_i : A_i \hookrightarrow A_1 \otimes A_2$ denote the inclusion which is a rational model for the projection $\pr_i : X_1 \times X_2 \to X_i$.
Remark that $ \wedge V_1 \otimes \wedge V_2 \cong \wedge (V_1 \oplus V_2)$ is a Sullivan model for $Y_1 \times Y_2$; that is, the Brown-Szczarba model for $\Map(X_1 \times X_2, Y_1\times Y_2)$ is of the form $\wedge ( (V_1 \oplus V_2) \otimes (A_1 \otimes A_2)_* )$.
Let $\psi : \wedge V_2 \to A_2$ be a Sullivan representative for $g$.
Then, we define the CDGA morphism $\widehat{\psi}$ as the following composite;
\[
\xymatrix@C15pt@R15pt{
\wedge ( (V_1 \oplus V_2) \otimes (A_1 \otimes A_2)_* )
\ar[r]^-{\rho}_-{\cong}
&
\wedge (\wedge V_1 \otimes \wedge V_2 \otimes (A_1 \otimes A_2)_*)/{\mathcal I}
\ar[d]_-{\wedge (1\otimes \psi \otimes 1)}
\\
&
\wedge (\wedge V_1 \otimes A_2 \otimes (A_1 \otimes A_2)_*)/{\mathcal I}
\ar[d]_-{\widetilde{\eta}}
\\
\wedge (V_1 \otimes A_{1*})
&
\wedge (\wedge V_1 \otimes A_{1*})/{\mathcal I},
\ar[l]_-{\rho^{-1}}^-{\cong}
}
\]
where $\widetilde{\eta}$ is the CDGA morphism which is induced by the natural isomorphism $\zeta : (A_1\otimes A_2)_* \stackrel{\cong}{\rightarrow} A_{1*}\otimes A_{2*}$ and the paring $\eta : A_2 \otimes A_{2*} \to \Q$.

\begin{lem}\label{BSmodel:times}
The morphism $\widehat{\psi}$ is a rational model for $- \times g$.
\end{lem}

\begin{proof}
First, we see that the map $- \times g$ coincides with the composite
\begin{equation}\label{geom:times_g}
\xymatrix@C10pt@R20pt{
\Map (X_1 , Y_1)
\ar[r]^-{(1 , c_g)}
&
\Map (X_1 , Y_1)
\times \Map (X_2 , Y_2)
\ar[ld]_-{\pr_1^*  \times  \pr_2^*}
\\
\Map (X_1 \times X_2 , Y_1)
\times
\Map (X_1 \times X_2 , Y_2)
\ar[r]^-{\cong}
&
\Map (X_1 \times X_2 , Y_1 \times Y_2),
}
\end{equation}
where $c_g : \Map (X_1 , Y_1) \to \Map (X_2 , Y_2)$ is the constant map at $g$.
Let $\Map_g (X_2 , Y_2)$ be the connected component of $\Map (X_2 , Y_2)$ containing $g$.
We also see that $c_g$ is regarded as the composite
\[
\xymatrix{
\Map (X_1 , Y_1)
\ar[r]
&
 {\rm pt}
\ar[r]^-{c_g}
&
\Map_g (X_2 , Y_2)
\ar[r]^-{\inc}
&
\Map (X_2 , Y_2).
}
\]
It follows from the rational model for the inclusion $\Map_g (X_2 , Y_2) \hookrightarrow \Map (X_2 , Y_2)$ described in  \cite[Proposition 4.2, Theorem 4.5]{BM06} and \cite[Remark 3.4]{HKO08} that the morphism
\[
{\mathcal M}_{c_g} : \wedge (V_2 \otimes A_{2*}) \to \wedge (V_1 \otimes A_{1*})
\] defined by
$v_2 \otimes \alpha_2 \mapsto (-1)^{|v_2||\alpha_2|}\alpha_2 (\psi (v_2))$
for $v_2 \otimes \alpha_2 \in V_2 \otimes A_{2*}$ is a rational model of $c_g$.
We choose the inclusion $\iota_i$ into the $i$th factor as a model for the projection $\pr_i$ in the $i$th factor.  Then the description \eqref{geom:times_g}
and the naturality of the Brown-Szczarba models shows that the composite
\begin{equation}\label{model:times_g}
\xymatrix@C10pt@R15pt{
\wedge ( (V_1\oplus V_2)\otimes (A_1 \otimes A_2)_* )
\ar[r]^-{\cong}
&
\wedge ( V_1\otimes (A_1 \otimes A_2)_* ) \otimes \wedge ( V_2\otimes (A_1 \otimes A_2)_* )
\ar[ld]_-{\iota_1^* \otimes \iota_2^*}
\\
\wedge ( V_1\otimes A_{1*})  \otimes \wedge ( V_2 \otimes A_{2*} )
\ar[r]^-{ 1 \cdot {\mathcal M}_{c_g}}
&
\wedge ( V_1\otimes A_{1*})
}
\end{equation}
is a model for the map $- \times g$.
For any element $\alpha \in (A_1\otimes A_2)_*$, we may write
\[
\zeta (\alpha) =  \alpha \circ \iota_1 \otimes 1^* + 1^* \otimes \alpha \circ \iota_2 + \widetilde{\alpha}
\]
with $\widetilde{\alpha} \in (A_1^+)_* \otimes (A_{2}^+)_*$.
Observe that $1\otimes \widetilde{\alpha}_1$ is zero in $\wedge (\wedge V \otimes A_{1*})/{\mathcal I}$ for any $\widetilde{\alpha}_1 \in (A_1^+)_*$.
Therefore, we can check that the composite \eqref{model:times_g} coincides with $\widehat{\psi}$ and the proof is complete.
\end{proof}

We are ready to construct a rational model for the adjoint map
\begin{equation}\label{def:ad}
\ad : \Map (X_1, \Map_f (X_2 , Y)) \to \Map (X_1\times X_2, Y)
\end{equation}
defined by $\ad(g)(x_1, x_2) = g(x_1)(x_2)$
for $f:X_2 \to Y$, $g :X_1 \to \Map_f (X_2 , Y)$ and $x_i \in X_i$.
In order to apply the rational models due to Brown and Szczarba to our objects, we need to consider the connected component $\Map_f (X_2 , Y)$ containing $f$.

Let $\varphi$ be a Sullivan representative for $f$, $\varphi : \wedge (V\otimes A_{2*}) \to \Q$ the  augmentation induced by $\varphi$ and $\wedge S_\varphi$ the Brown-Szczarba model for $\Map_f (X_2 , Y)$ mentioned above.
Since $\Map_f (X_2 , Y)$ is a connected and nilpotent space of finite type; see \cite{HMR75}, by applying the construction of the  Brown-Szczarba model to $X_1$ and
$\Map_f (X_2 , Y)$, we have a model for
$\Map (X_1, \Map_f (X_2 , Y))$ of the form $\wedge (S_\varphi \otimes A_{1*})$.

\begin{prop}\label{model:ad}
The morphism ${\mathcal M}_{\ad} : \wedge (V\otimes (A_1 \otimes A_2)_*) \to \wedge ( S_\varphi \otimes A_{1*})$ of CDGAs defined by
\[
{\mathcal M}_{\ad} (v \otimes \alpha )
  =
  \sum (-1)^{|\alpha_1||\alpha_2|}\rho^{-1}\left( {\mathcal M}_{\inc}(v\otimes \alpha_2) \otimes \alpha_1 \right)
\]
is a rational model for the adjoint map $\ad$ in \eqref{def:ad}, where $v\in V$, $\alpha \in  (A_1 \otimes A_2)_*$ and $\zeta (\alpha)= \sum \alpha_1 \otimes \alpha_2$; see the paragraph before Lemma \ref{BSmodel:times} for the maps $\rho$ and $\zeta$.
\end{prop}

\begin{proof} The map $\ad$ fits in the commutative diagram
\[
\xymatrix@C3pt@R12pt{
\Map (X_1, \Map_f (X_2 , Y))
\ar[rd]^-{ - \times id}
\ar[rr]^-{\ad}
&
&
\Map (X_1 \times X_2 , Y)
\\
&
\Map (X_1 \times X_2 , \Map_f (X_2 , Y) \times X_2),
\ar[ru]^-{\ev_*}
&
}
\]
where $\ev : \Map_f (X_2 , Y) \times X_2 \to Y$ is the evaluation map.
The result  \cite[Theorem 1.1]{BM06} enables us to obtain a Sullivan model
\[
{\mathcal M}_{\ev} : \wedge V \to \wedge S_\varphi \otimes A_2
\]
for $\ev$ defined by
${\mathcal M}_{\ev}(v) = \sum_i (-1)^{|a_i|} {\mathcal M}_{\inc}(v\otimes a_i^*)\otimes a_i$,
where $\{ a_i \}$ is a basis of $A_2$ and $\{ a_i^* \}$ is the dual basis of $A_{2*}$.
By the surjective trick \cite[p.148]{FHT1}, there exist a Sullivan model $\wedge W$ for $X_2$ and a surjective Sullivan representative $\sigma : \wedge W \to A_2$ for the identity on $X_2$.
The lifting lemma \cite[Lemma 12.4]{FHT1} shows that there exists a morphism ${\mathcal M}'_{\ev} : \wedge V \to \wedge S_\varphi \otimes \wedge W$ such that $(1\otimes \sigma) \circ {\mathcal M}'_{\ev} = {\mathcal M}_{\ev}$.
Lemma \ref{BSmodel:times} is applicable  to the map $- \times id$. Thus we see that the composite
\begin{equation}\label{model:ad2}
\xymatrix@C10pt@R15pt{
\wedge (V \otimes (A_1\otimes A_2)_*)
\ar[r]^-{\rho}_-{\cong}
&
\wedge (\wedge V \otimes (A_1\otimes A_2)_*)/{\mathcal I}
\ar[d]^-{\wedge ({\mathcal M}'_{\ev} \otimes 1)}
\\
&
\wedge (\wedge S_\varphi \otimes \wedge W \otimes (A_1\otimes A_2)_*)/{\mathcal I}
\ar[d]^-{\rho^{-1}}_-{\cong}
\\
\wedge ( S_\varphi \otimes A_{1*})
&
\wedge ( (S_\varphi \oplus \wedge W) \otimes (A_1\otimes A_2)_*)
\ar[l]_-{\widehat{\psi}}
}
\end{equation}
is a model for $\ad$, where $\widehat{\psi}$ is the morphism of CDGAs defined in the paragraph before Lemma \ref{BSmodel:times}.
Explicitly, we compute
\begin{align}
&\widehat{\psi} \circ \rho^{-1} \circ \wedge ({\mathcal M}'_{\ev} \otimes 1) \circ \rho (v\otimes \alpha)
\\
=& \rho^{-1} \circ \widetilde{\eta} \circ \wedge ({\mathcal M}_{\ev} \otimes 1)(v\otimes \alpha)
\\
=& \rho^{-1} \circ \widetilde{\eta} \left( \sum_i (-1)^{|a_i|} {\mathcal M}_{\inc}(v\otimes a_i^*)\otimes a_i \otimes \alpha \right)
\\
=& \rho^{-1} \left( \sum_i \sum (-1)^{|a_i| + |\alpha_2|(|a_i| + |\alpha_1|)} {\mathcal M}_{\inc}(v\otimes a_i^*)\cdot \alpha_2(a_i)  \otimes \alpha_1  \right)
\\
=&  \sum (-1)^{|\alpha_1||\alpha_2|} \rho^{-1}\left( {\mathcal M}_{\inc}(v\otimes \alpha_2)  \otimes \alpha_1 \right).
\end{align}
Therefore, the model \eqref{model:ad2} for the adjoint map $\ad$ is nothing but the morphism ${\mathcal M}_{\ad}$ mentioned in the assertion.
\end{proof}

\section{An equivariant version of the Lie derivative $L$}\label{app:bar{L}}

\subsection{A geometric construction of a Lie derivative}\label{subsect:geom_bar{L}}
In this section, we assume
that the underlying field is of arbitrary characteristic. 

We discuss an equivariant cohomology (cyclic homology) version of the Lie derivative  $L$ that we consider in \cref{sect:geom_L_e}.
We begin by recalling a morphism of Lie algebras related to the Hochschild homology and the cyclic homology of an algebra.
Let $A$ be an unital algebra over a commutative ring $k$. For a derivation $D$ on $A$, we define a map $L_D$ on the Hochschild complex $C_*(A)$ by
\[
L_D(a_0, ..., a_n) = \sum_{i\geq 0} (a_0, .., a_{i-1}, Da_i, a_{i+1}, ..., a_n).
\]
We also recall the Hochschild cohomology $HH^*(A, A)$ of $A$. In particular, the first cohomology $HH^1(A, A)$ is
isomorphic to $\Der(A) /\{\text{inner derivations} \}$ as a $k$-module.
Then, we have

\begin{prop}\label{prop:Loday'sDer} {\em (\cite[4.1.6 Corollary]{Loday})}
  There are well-defined homomorphisms of Lie algebras $[D] \mapsto L_D$:
  \[
    HH^1(A, A) \to \End_k(HH_n(A)) \ \  \text{and} \ \ \ HH^1(A, A) \to \End_k(HC_n(A)).
  \]
\end{prop}

In the body of this manuscript, we discuss a geometric description of the Lie derivative  on the endomorphism algebra of the Hochschild homology of a DGA. The above result motivates us to consider its cyclic version. In this section, we deal with the topics. As a consequence, our main theorem, \cref{thm:LieAlg-New} below is obtained.

We work on the category of compactly generated spaces \cite{Steenrod}
or the category $\mathsf{NG}$ of numerically generated spaces, which is obtained by adjoint functors between the category of topological spaces and that of diffeological spaces; see \cite{SYH}. Thus, we can consider a space in such a Cartesian closed category without changing the weak homotopy type. Observe that the category $\mathsf{NG}$ is also complete and cocomplete.



Let $X$ be a simply-connected space of finite type and $\aut X$ the monoid of self-homotopy equivalences on $X$. We recall that
the homotopy group $\pi_*(\aut X)$ is a Lie algebra with the Samelson product; see \cite[Chapter III]{Whitehead}.
For an element $\theta$ in the homotopy group $\pi_n(\aut X)$ for $n > 1$,
we define a map $u_\theta$ by the composite
\[
  \xymatrix@C18pt@R20pt{
    u_\theta := L( \ ) \circ inc \circ \theta : S^n \ar[r] & \aut X \ar[r] & \Map(X, X) \ar[r]^-{L} & \Map(LX, LX),
  }
\]
where $inc$ denotes the inclusion and $L$ is the map which assigns $Lf : LX \to LX$ defined by ${Lf}(l) = f\circ l$ to a map $f :  X \to X$.
Then, the adjoint map $ad(u_\theta) : S^n \times LX \to LX$ gives rise to the derivation
\[
  \xymatrix@C20pt@R20pt{
    L'_\theta : H^*(LX) \ar[r]^-{(ad(u_\theta))^*} & H^*(S^n)\otimes H^*(LX)  \ar[r]^-{\int_{S^n}} &H^{*-n}(LX)
  }
\]
on the cohomology $H^*(LX)$, where $\int_{S^n}$ denotes the integration along the fiber.
The map $L_\theta$ in (\ref{defn:L_geom}) is regarded as the composite $\int_{S^n}\circ (s\times 1)^*\circ L(ad(\theta))^*$, where $s : S^n \to LS^n$ is the section of the evaluation map $\ev_0$ defined by $s(x)(t) = x$ for $x \in S^n$ and $t \in S^1$. Since $ad(u_\theta) = L(ad(\theta))\circ (s\times 1)$, it follows that $L'_\theta$ coincides with $L_\theta$ in (\ref{defn:L_geom}). In what follows, we may write $L_\theta$ for $L'_\theta$.


Observe that the adjoint map $ad(u_\theta) : S^n \times LX \to LX$ is an $S^1$-equivariant map, where the $S^1$-action on  $S^n$ is defined to be trivial. Thus, we have a map
$\overline{ad(u_\theta)\times_{S^1}1} : (S^n\times LX)\times_{S^1}ES^1 \to LX\times_{S^1}ES^1$ between the Borel constructions.
Therefore, the same construction as that of
$L_\theta$ with the integration enables us to obtain a derivation
\[
  \xymatrix@C20pt@R20pt{
    \overline{L}_\theta : H^*_{S^1}(LX) \ar[r] &H^{*-n}_{S^1}(LX)
  }
\]
of degree $-n$.

The assertion below describes geometric counterparts of the morphisms of Lie algebras described in \cref{prop:Loday'sDer}.

\begin{thm}\label{thm:LieAlg} The map
    $\overline{L}_{( \ )} : \pi_*(\aut X) \to \Der_{*}(H^*_{S^1}(LX))$
  is a morphism of Lie algebras.
\end{thm}

\begin{proof}[Proofs of Theorems \ref{thm:LieAlgL} and \ref{thm:LieAlg}]
As mentioned above, the map $ad(u_\theta)$ is an $S^1$-equivariant map. Then, the operation $L'_\theta$ commutes with the BV operator. We have  Theorem \ref{thm:LieAlgL} \cref{item:LieAlgL-ii}.

In order to prove Theorem \ref{thm:LieAlgL} \cref{item:LieAlgL-i} and \ref{thm:LieAlg},
  we first recall that
  \[
    L_\theta = \int_{S^n} \circ H^*(ad(L_*(\theta))) \ \text{ and}  \ \
    \overline{L}_\theta = \int_{S^n} \circ H^*(ad(L_*(\theta))\times_{S^1} 1_{ES^1})
  \]
  for $\theta \in \pi_n(\aut X)$. We may write $ad(L_*(\theta))^\sigma$ for
  $ad(L_*(\theta))\times_{S^1} 1_{ES^1}$.
  The map $L$  mentioned above induces a homomorphism
  $L_* : \pi_*(\aut X) \to \pi_*({\rm aut}_1 (LX))$.  Let $\theta_1$ and $\theta_2$ be
  homotopic maps which represent an element in $\pi_n(\aut X)$. Then, we see that the maps
  $ad(L_*(\theta_1))$ and $ad(L_*(\theta_2))$ from $S^n \times LX$ to $LX$ are homotopic with an $S^1$-equivariant homotopy. This implies that ${L}_\theta$  and $\overline{L}_\theta$ are well defined.
  In what follows, we prove that $\overline{L}_{( \ )}$ is a morphism of Lie algebras. The same argument as that for $\overline{L}_{( \ )}$ is applicable to showing the result on $L_{( \ )}$.
  As a consequence, we have Theorem \ref{thm:LieAlgL} \cref{item:LieAlgL-i}.

  We apply the same strategy as that for \cite[Lemma 4.1 and Theorems 3.6, 4.2 and 4.3]{FLS}.
  In order to prove that $\overline{L}_{( \ )}$ is a homomorphism, we consider a diagram
  \[
    \xymatrix@C100pt@R15pt{
      (S^n \vee S^n) \times LX\times_{S^1} ES^1 \ar[r]^-{(ad(L_*(\theta))^\sigma\mid ad(L_*(\theta'))^\sigma)} & LX\times_{S^1} ES^1 \\
      S^n \times LX\times_{S^1} ES^1 \ar[u]^{\tau \times 1}\ar[ru]_-{\hspace{1.8cm } ad(L_*(\theta + \theta'))^\sigma= ad(L_*(\theta) + L_*(\theta'))^\sigma}
    }
  \]
  in which $(ad(L_*(\theta))^\sigma\mid ad(L_*(\theta'))^\sigma) \circ i_1 = ad(L_*(\theta))^\sigma$ and $(ad(L_*(\theta))^\sigma\mid ad(L_*(\theta'))^\sigma) \circ i_2 = ad(L_*(\theta'))^\sigma$, where $\tau$ is the pinch map and $i_j$ is the map induced by the inclusion $S^n \to S^n \vee S^n$ in the $j$ factor.
  Then, it follows that the horizontal arrow assigns $\chi + u\overline{L}_\theta(\chi) + v\overline{L}_{\theta'}(\chi)$ to an element
  $\chi \in H^*(LX\times_{S^1}ES^1)$, where $(u, 0)$ and $(0, v)$ denotes the generators of $H^n(S^n\vee S^n)$.
  The definition of the summation in $\pi_*(\text{aut}_1(LX))$ implies that the diagram above is commutative. Moreover, by definition, the slant arrow induces $\overline{L}_{\theta + \theta'}$. This yields that $\overline{L}_\theta + \overline{L}_{\theta'} = \overline{L}_{\theta + \theta'}$.

  Let $\overline{\theta}$ be the inverse of $\theta$ in $\pi_*(\aut X)$
  with respect to the multiplication of the monoid $\aut X$. Since $\overline{L}_{( \ )}$ is a homomorphism, it follows that
  $\overline{L}_{\overline{\theta}} = - \overline{L}_{\theta}$.

  We recall the Samelson product $\langle \ , \ \rangle$ on the homotopy group $\pi_*(\aut X)$. For elements $\theta : \pi_p(\aut X)$
  and $\theta' \in \pi_q(\aut X)$, the product is induced by the map
  $\gamma : S^p \times S^q \to \aut X$ defined by $\gamma(x, y) = \theta(x)\circ \theta'(y)\circ \overline{\theta}(x) \circ \overline{\theta'}(y)$.
  Then, we have
  \[
    (L\circ \gamma)(x, y) = L_*(\theta)(x)\circ L_*(\theta')(y) \circ L_*(\overline{\theta})(x) \circ L_*(\overline{\theta'})(y).
  \]
  Observe that $L$ is a morphism of monoids. Therefore, the adjoint $\Gamma$ to $L\circ \gamma$ fits in the commutative diagram
  \[
    \xymatrix@C55pt@R12pt{
      S^p \times S^q\times LX \times_{S^1} ES^1 \ar[rdd]_{\Gamma \times_{S^1} 1} \ar[r]^-{\text{Diag}\times \text{Diag} \times_{S^1} 1} & S^p \times S^p\times S^q \times S^q \times LX \times_{S^1} ES^1\ar[d]^{1_{S^p} \times T \times 1_{S^q} \times 1_{LX}\times_{S^1} 1_{ES^1}} \\
      &  S^p \times S^q\times S^p \times S^q \times LX \times_{S^1} ES^1\ \ar[d]^{[F,G]\times _{S^1}1_{ES^1}}\\
      & LX \times_{S^1} ES^1
    }
  \]
  where $\text{Diag}$ is the diagonal map, $T$ denotes the transposition and $[F, G]$ is defined by the composite
  \[
    ad(L_*(\theta))\circ (1_{S^p}\times ad(L_*(\theta')))\circ (1_{S^p\times S^q}\times  ad(L_*(\overline{\theta})))
    \circ  (1_{S^p\times S^q\times S^p}\times  ad(L_*(\overline{\theta'}))).
  \]
  The commutativity follows from the same consideration as in the proof of \cite[Theorem 4.3]{FLS}.
  Moreover, the same computation as in \cite[page 394]{FLS} works well on homology. It turns out that
  $\overline{L}_{\langle \theta , \theta' \rangle} =
  \overline{L}_\theta\overline{L}_{\theta'}-(-1)^{pq}\overline{L}_{\theta'}\overline{L}_\theta$.
  This completes the proof of \cref{thm:LieAlg}.
\end{proof}

\subsection{An algebraic construction of $\overline{L}$}
In what follows, we assume that the underling field is rational.
The assertion below shows that the geometric derivations $L_{( \ )}$ and $\overline{L}_{( \ )}$ are related to the
Loday's derivations $_aL_{( \ )}$  and $_a\overline{L}_{( \ )}$ in \cref{prop:Loday'sDer}, respectively.
Observe that $_aL_{\theta}$  is the derivation $L_\theta$ in Definition  \ref{defn:derivationSullivan}.

%
\begin{proof}[Proof of Proposition  \ref{prop:LieAlg-New}] The standard algebraic model for the evaluation map $\ev : LX \times S^1 \to X$ plays an important role in our proof; see \cite{V-S} for the model for $\ev$.
  We consider the following commutative diagram consisting of continuous maps
  \begin{equation}\label{eq:adjoints_top}
    \xymatrix@C20pt@R18pt{
      \Map(S^n, \Map(X, X))\ar[r]^-{L_*} \ar[d]_{ad}^{\cong}& \Map(S^n, \Map(LX, LX))\ar[d]^{ad_1}_{\cong} \\
      \Map(S^n\times X, X) \ar[r]^-{ad(L_*)} \ar[rd]_-{\psi:= ad_2\circ ad(L_*)}& \Map(S^n \times LX, LX) \ar[d]^{ad_2}_{\cong} \\
      & \Map(S^n\times LX \times S^1, X). }
  \end{equation}
  It follows that $ad_1(u_\theta) = ad(L_*)(ad(\theta))$ for
  $\theta \in \Map(S^n, \Map(X, X))$
  and $\psi(\phi) = \phi\circ (1\times \ev)$ for $\phi \in  \Map(S^n\times X,  X)$.
  In what follows, we may assume that $X$ is a rational space.
  Then, we have the following diagram for the homotopy sets
  \[
    \xymatrix@C15pt@R15pt{
      \pi_n(\aut X) \ar[r]^k & [S^n \times X, X] \ar[r]^-{ad(L_*)} \ar[d]_{\mu}^{\cong} &  [S^n \times LX, LX] \ar[r]^-{ad_2}_-{\cong} \ar[d]_{\mu}^{\cong}&
      [S^n\times LX \times S^1, X]  \ar[d]_{\mu}^{\cong}\\
      & [{\mathcal M}_X, {\mathcal M}_{S^n\times X}] &  [{\mathcal M}_{LX}, {\mathcal M}_{S^n\times LX}]  &  [{\mathcal M}_{X},  {\mathcal M}_{S^n\times LX\times S^1}]
    }
  \]
  which are given by the sets of continuous maps mentioned above. Here $k$ is induced by the adjoint $ad$ mentioned above and ${\mathcal M}_Y$ denotes a minimal Sullivan model for a space $Y$ and $\mu$ is the Sullivan--de Rham correspondence between rational spaces and CDGAs; see, for example,
  \cite{BG}.
  We use the same notation for a map as that for its homotopy class.

  We may replace ${\mathcal M}_{S^n\times X}$ with the CDGA $H^*(S^n;\Q)\otimes {\mathcal M}_X$; see \cite[Proposition 12.9]{FHT1}. Then we write
  $(\mu\circ k)(\theta') = 1\otimes 1_{{{\mathcal M}_X}} + \iota \otimes \theta$, where $\iota$ is the generator of  $H^n(S^n;\Q)$. Observe that, by definition,
  $\Phi(\theta') = \theta$ for the map $\Phi$ in \cref{prop:LieAlg-New}. In order to prove Proposition  \ref{prop:LieAlg-New}, it suffices to show

  \begin{lem}\label{lem:adjoint_alg} For $\theta' \in \pi_n(\text{\em aut}_1(X))$, one has
    $(\mu\circ ad(L_*) \circ k)(\theta') =
    1\otimes 1_{{\mathcal M}_{LX}} + \iota{}_aL_\theta$.
  \end{lem}

  \noindent
  In fact,
  applying the integration $\int_{S^n}$  to the equality in \cref{lem:adjoint_alg} on the cohomology yields the commutativity of the diagram in \cref{prop:LieAlg-New}.
\end{proof}

\begin{proof}[Proof of \cref{lem:adjoint_alg}] We consider the adjoint map $ad_2$. The uniqueness of the adjoint correspondence shows that if we have a morphism $L(\theta')$ of CDGAs which makes the following triangle
  \[
    \xymatrix@C15pt@R10pt{
      H^*(S^n;\Q)\otimes {\mathcal M}_{LX}\otimes \wedge(t) & &  {\mathcal M}_{LX}\otimes \wedge(t) \ar[ll]_-{L(\theta')\otimes 1} \\
      & {\mathcal M}_X =(\wedge V, d) \ar[ur]_{\mu(\ev)} \ar[ul]^-(0.6){
      \mu(1\times \ev)\circ \mu(k(\theta'))=\mu (k(\theta') \circ (1\times \ev))\ \ \ \ \ \ \ \ \ \ \ } &
    }
  \]
  commutative up to homotopy, the map $L(\theta')$ is nothing but the map $(\mu\circ ad(L_*) \circ k)(\theta')$. In fact,
    the map $ad_2$ assigns the realization $|L(\theta')|$
    of $L(\theta')$ to the realization $|\mu (k(\theta') \circ (1\times \ev))|$, which is homotopic to
    $ad_2\circ ad(L_*)(\theta')=k(\theta') \circ (1\times \ev))$.
    Observe that the equality follows from the commutativity of the diagram (\ref{eq:adjoints_top}).
    Then we have
    $ad_2(|L(\theta')|)\simeq |\mu (k(\theta') \circ (1\times \ev))|\simeq  ad_2(ad(L_*)(k(\theta')))$. The injectivity of the map $ad_2$ yields that $|L(\theta')| \simeq ad(L_*)(k(\theta'))$.  This implies that the map $L(\theta')$ is a model for $ad(L_*)(k(\theta'))$; that is, $L(\theta')= (\mu\circ ad(L_*) \circ k)(\theta')$.

  We recall the Sullivan model ${\mathcal M}_{LX} = \HochschildModel$ described in \cref{sect:preliminaries}.
  Moreover, we may choose a model $\mu(\ev)$ for the evaluation map so that $\mu(\ev)(\omega) = \omega \otimes 1 + (-1)^{\deg{\omega}-1}s\omega \otimes t$ for $\omega \in V$; see \cite{V-S}.
  Since $(\mu\circ k)(\theta') = 1\otimes 1_{{\mathcal M}_X} + \iota \otimes \theta$, it follows from the commutativity for the triangle that
  \begin{eqnarray*}
    &&L(\theta')\omega \otimes 1 + (-1)^{\deg{\omega}-1}L(\theta')(s\omega)\otimes t \\
    &=& 1\otimes (\omega\otimes 1 + (-1)^{\deg{\omega}-1}s\omega \otimes t) + \iota \otimes (\theta(\omega)\otimes 1 + (-1)^{\deg{\theta(\omega)}-1}s\theta(\omega) \otimes t )
  \end{eqnarray*}
  for $\omega \in V$.  Therefore, we see that $L(\theta')\omega = 1\otimes \omega + \iota \theta(\omega)$ and
  $L(\theta')(s\omega)= 1\otimes s\omega + (-1)^{\deg{\theta}}s\theta(\omega)$. The definition of ${}_aL_\theta$ shows that $L(\theta') =1\otimes 1_{{\mathcal M}_{LX}} + \iota{}_aL_\theta$. This completes the proof.
\end{proof}


Next we review the relationship with cyclic homology. Let \(\mixed  = (\mixed , d, B) \) be a non-negatively graded mixed complex.
We introduce a variable \(u\) of degree $2$ and consider the graded module \(\mixed [u] = \mixed\otimes\Q[u]\).

\begin{defn}
  The \textit{cyclic complex}
  \(\cyclic{\mixed}\) of \(\mixed\) is the complex \((\mixed [u], \dcyc)\),
  where \(\dcyc\) is the \(\Q[u]\)-linear map defined by \(\dcyc = \done + u\dtwo\).
  Its cohomology will be called the \textit{cyclic cohomology} of \(\mixed\)
  and denoted by \(HC(\mixed)\).
\end{defn}


We recall the mixed DGA $(\HochschildModel, d, s)$ mentioned in Section \ref{sect:preliminaries}. With the model, the minimal Sullivan model $\cyclicModel$ of the orbit space $ES^1\times_{S^1}LX$ is defined  by $\cyclicModel:=(\HochschildModel[u], d+us)$; see \cite[Theorem A]{V-B1}. Thus we have an isomorphism $H(\cyclicModel) \cong H^{*} (ES^1\times_{S^1}LX; \Q)$.
Observe that \(\cyclic{\HochschildModel}\) is nothing but the complex $\cyclicModel$ defined above and then
\(H(\cyclic{\HochschildModel})\)
is isomorphic to the cyclic homology \cite{BV88} of \((\wedge V, d)\).

  Let \((\lie, \cce, \ccl, \ccs, \cct)\) be a homotopy Cartan calculus on \(\mixed\).
  For \(\theta\in\lie\), define \({}_a\cycl_\theta\in\End(\cyclic{\mixed})\)
  by extending \(\ccl_\theta\)  to a \(\Q[u]\)-linear map.
  This gives a linear map \({}_a\cycl\colon\lie\to\End(\cyclic{\mixed})\).

\begin{lem}[{\cite[Lemmas 3.4 and 3.10]{FK}}]
  The map \({}_a\cycl\colon\lie\to\End(\cyclic{\mixed})\)
  is a morphism of dg Lie algebras.
\end{lem}

For the homotopy Cartan calculus in Proposition \ref{prop:CartanCalculusSullivanModel}, we obtain the morphism ${}_a\overline{L}_{( \ )} : \Der (\wedge V) \to \Der (\cyclicModel )$ defined by $_a\overline{L}_\theta = {}_aL_\theta \otimes 1_{\Q[u]}$ on $\cyclicModel$ for $\theta \in \Der (\wedge V)$.
We recall the cobar-type Eilenberg-Moore spectral sequence (EMSS) in \cite[Theorem 7.5]{KNWY} converging
to the string cohomology $H^*_{S^1}(LX; \Q)$ with
\[
E_2^{*,*} \cong \text{Cotor}_{H^*(S^1; \Q)}^{*,*}(H^*(LX; \Q), \Q).
\]
Let $\{F^p\}_{\geq 0}$ be the decreasing filtration of $H^*_{S^1}(LX;\Q)$ associated with the EMSS.

\begin{thm}\label{thm:LieAlg-New}
   There exists a commutative diagram
  \[
    \xymatrix@C20pt@R20pt{
      \pi_*(\aut X)\otimes \Q \ar[r]^{\overline{L}_{( \ )}}  \ar[d]_{\cong}^{\Phi}& \Der_{*}(H^*_{S^1}(LX;\Q)) \ar[d]^{\cong}\\
      H_*(\Der(\wedge V)) \ar[r]_{{}_a\overline{L}_{( \ )}} & \Der_{*}(H^*(\cyclicModel))
    }
  \]
  modulo the filtration of the EMSS in the sense that $({}_a\overline{L}_{\theta} - \overline{L}_\theta )(F^p) \subset F^{p+1}$ for
  $\theta$ in $\pi_*(\aut X)\otimes \Q$ and $p \geq 0$, where $\{F^p\}$ is the filtration of $H^*_{S^1}(LX;\Q)$ associated with the EMSS mentioned above.
\end{thm}

\begin{proof}
  The key to proving the result is that the {\it projection} of a model for the derivation $\overline{L}_\theta$ on $H^*_{S^1}(LX;\Q)$ is the model ${}_aL_\theta$ the derivation on $\HochschildModel$ considered in the proof of \cref{prop:LieAlg-New}. Let $u_\theta$ be the map stated in section \ref{subsect:geom_bar{L}}. Consider the commutative diagram
  \[
    \xymatrix@C18pt@R5pt{
      & (S^n \times LX) \times_{S^1} ES^1 \ar[ld] \ar[dd]^{ad_1(u_\theta) \times 1_{ES^1}=\overline{v_\theta}} &  S^n \times LX \ar[l] \ar[dd]^{ad_1(u_\theta)=v_\theta}\\
      BS^1 & &  \\
      & LX \times_{S^1} ES^1 \ar[lu]  &  LX  \ar[l]
    }
  \]
  whose row sequences are the fibrations associated with the universal $S^1$-bundle $S^1 \to ES^1 \to BS^1$. For simplicity, we put $v_\theta := ad_1(u_\theta)$ and $\overline{v_\theta}:= ad_1(u_\theta) \times 1_{ES^1}$. Observe that $v_\theta$ is an $S^1$-equivariant map. We moreover consider the relative Sullivan models for the fibrations and the morphism between them described in \cite[(15.9) pages 204--205]{FHT1}. Then we have a commutative diagram
  \[
    \xymatrix@C18pt@R5pt{
      & H^*(S^n;\Q) \otimes  \cyclicModel  \ar[r] & H^*(S^n;\Q) \otimes \HochschildModel \\
      \Q[u] \ar[ru] \ar[rd] & &  \\
      &  \cyclicModel \ar[uu]_{{\mathcal M}(\overline{v_\theta})} \ar[r] &  \HochschildModel  \ar[uu]_{\overline{{\mathcal M}(\overline{v_\theta})}}
    }
  \]
  in which ${\mathcal M}(\overline{v_\theta})$ and $\overline{{\mathcal M}(\overline{v_\theta})}$ are algebraic models for $\overline{v_{\theta}}$ and $v_\theta$, respectively.
  Then \cref{lem:adjoint_alg} implies that $\int_{S^n}\circ \overline{{\mathcal M}(\overline{v_\theta})}$ is chain homotopic to ${}_aL_\theta$; that is, there exists a homotopy $h_\theta'$ of degree $-1$ with $\int_{S^n}\circ \overline{{\mathcal M}(\overline{v_\theta})} - {}_aL_\theta = d h_\theta' + h_\theta' d$ in $ \HochschildModel$.
  Observe that ${\mathcal M}(\overline{v_\theta})$ is a morphism of $\Q[u]$-modules. Thus we see that for $x \in  \widetilde{F}^p:= \cyclicModel \cdot \Q^{\geq p}[u]$,
  \begin{equation}\label{eq:onE}
    \left( \int_{S^n} \circ {\mathcal M}(\overline{v_\theta}) - {}_aL_\theta\otimes1_{\Q[u]} \right)x = (d h_\theta + h_\theta d)x + \alpha_{\theta, x}
  \end{equation}
  with $h_\theta = 1_{H^*(S^n)}\otimes h'_\theta$ and for some $\alpha_{\theta, x}$ in $\widetilde{F}^{p+1}$.
  By construction, the filtration $\{\widetilde{F}^p\}_{p\geq 0}$ gives rise to the EMSS that we deal with. Moreover, the filtration $\{F_p\}_{p\geq 0}$ associated with the EMSS is induced by $\{\widetilde{F}^p\}_{p\geq 0}$.

  Suppose that  $x$ is a cocycle with respect to the differential $D:= d + us$ of $\cyclicModel$.
  We may write $x= (x^0, x^1, ... )$. By applying $D$ to the both sides of the equality \cref{eq:onE}, we have
  $0 = D(d h_\theta + h_\theta d)x^0 + D\alpha_{\theta, x}'$, where
  $\alpha_{\theta, x}':= (d h_\theta + h_\theta d)x^{\geq 1} + \alpha_{\theta, x}$.
  Observe $\int_{S^n} \circ {\mathcal M}(\overline{v_\theta})$ and ${}_aL_\theta\otimes1_{\Q[u]}$ are cochain maps.
  Since $d x^0 =0$, it follows that $0=  us(d h_\theta x^0) + D\alpha_{\theta, x}'$. Thus we see that the element
  $-us  h_\theta x^0 + \alpha_{\theta, x}'$ is a cocycle in $\widetilde{F}^{p+1}$. It turns out that
  \[
    \left(\int_{S^n} \circ {\mathcal M}(\overline{v_\theta}) - {}_aL_\theta\otimes1_{\Q[u]}\right)x = Dh_\theta'x^0 + (-us  h_\theta x^0 + \alpha_{\theta, x}').
  \]
  By definition, we have  $\int_{S^n} \circ {\mathcal M}(\overline{v_\theta}) = \overline{L}_{\theta}$ and  ${}_aL_\theta\otimes1_{\Q[u]}= {}_a\overline{L}_{\theta}$ on the homology. This fact and the equality above yield the result.
\end{proof}

In a particular case, the square in \cref{thm:LieAlg-New} is commutative. To see this,
we first recall the BV-exactness of a space, which is a new homotopy invariant introduced in \cite{KNWY}.

\begin{defn}\label{defn:BV}(\cite[Definition 2.9]{KNWY})
A simply-connected space $X$ is {\it BV exact} if
$\im \widetilde{\Delta} = \ker \widetilde{\Delta}$ for the reduced BV operator
$\widetilde{\Delta} : \widetilde{H}_*(LX) \to  \widetilde{H}_{*+1}(LX)$.
\end{defn}

We observe that a formal space and a space which admits  positive weights are BV exact; see \cite[Assertion 1.2]{KNWY}.

\begin{cor}\label{cor:bv}
  Let $X$ be a BV exact space. Then the diagram in \cref{thm:LieAlg-New}
  is indeed commutative.
\end{cor}

\begin{proof}By assumption, the space $X$ is BV exact.
  Then it follows from \cite[Corollary 7.4]{KNWY} that the EMSS collapses at the $E_2$-term. Moreover, the result \cite[Lemma 7.5]{KNWY} implies that $F^p = 0$ for $p \geq 1$. This completes the proof.
\end{proof}

\begin{rem}
Let $M$ be a BV-exact manifold. 
Then the results \cite[Theorem 2.15 and Corollary 2.16]{KNWY} assert that the string bracket $[ \ , \ ]$ on $H_*^{S^1}(LM)$ is a restriction of the loop bracket $\{ \  , \  \}$ on the loop homology $\mathbb{H}_*(LM)$. 
More precisely, we have a commutative diagram
\[
\xymatrix@C50pt@R30pt{
H_*^{S^1}(LM; \K)^{\otimes 2} \ar[d]_{[ \ , \ ]} \ar[r]_-{\cong}^-{\Phi\otimes \Phi} &
(\ker \widetilde{\Delta} \oplus \K[u])^{\otimes 2} \ar[r]^-{(inc. \oplus 0)^{\otimes 2}} \ar[d]_{\pm\{ \ , \ \}}
 &  H_*(LM; \K)^{\otimes 2} \ar[d]_{\text{the loop product}}^{\bullet} \\
H_*^{S^1}(LM; \K)  \ar[r]_-{\cong}^-{\Phi}& (\ker \widetilde{\Delta} \oplus \K[u]) &
H_*(LM; \K), \ar[l]^-{\Delta}
}
\]
where $\pm \{a, b\} := (-1)^{|a|}\{a, b\}$ for $a, b \in \ker \widetilde{\Delta}$, $\deg u =2$, $\Delta$ is the BV operator,  $inc.$ denotes the inclusion and $\Phi$ is the isomorphism described in \cite[Theorem 2.15]{KNWY}.
%
\end{rem}



It seems that the representation $\overline{L}$ has a different property from that for $L$.

\begin{ex} (cf. \cref{ex:CP^2}) We determine explicitly the Lie representation $\overline{L}: \pi_*(\aut X)\otimes \Q \cong
H_*(\Der ({\mathcal M}_X))\to \Der_* (H_{S^1}^*(LX;\Q))$ for a simply-connected space $X$ whose rational cohomology is isomorphic to $\Q[x]/(x^{n+1})$ as an algebra, where  $n\geq 1$.
Let ${\mathcal M}_X$ be the minimal model for $X$. We see that
 ${\mathcal M}_X\cong (\wedge (x,y),d)$ in which $dx=0$ and $dy=x^{n+1}$.
Then the results  \cite[Theorem 2.2]{KY97} and
 \cite[Theorem 0.2]{KY00} yield that $$H^*_{S^1}(LX;\Q )\cong \oplus_{k\geq 0,1\leq j\leq n}\Q\{ \alpha (j,k)\}  \oplus \Q[u] $$
as an algebra, 
where $\alpha(j,k)=[x^{j-1}\bar{x}\bar{y}^k]$.
Moreover, we see that  $$H_*(\Der ({\mathcal M}_X))=\Q\{(y,1), (y,x), \cdots , (y,x^{n-1})\}.
$$
Since
 ${}_a\overline{L}_{(y,x^i) }(x^{j-1}\bar{x}\bar{y}^k)=kix^{j-1}\bar{x}x^{i-1}\bar{x}\bar{y}^{k-1}=0$ for $0< i< n$ and
 ${}_a\overline{L}_{(y,x^i) }(x^{j-1}\bar{x})=0$,
it follows that  ${}_a\overline{L}: H_*(\Der ({\mathcal M}_X))\to \Der_*(H_{S^1}^*(LX;\Q ))$ is trivial.
The space $X$ is formal and especially BV-exact. Thus Corollary \ref{cor:bv} yields that $\overline{L}=0$.
\end{ex}

We conclude this appendix with a brief discussion on the Lie representation $\overline{L}$ for a more general simply-connected space $X$, which is not necessarily BV exact. We consider a behavior of the operator $\overline{L}_{\theta}$ in the EMSS for each element $\theta \in \pi_*(\aut X)\otimes \Q$.
Let $\{E_r^{*,*}, d_r\}$ be the EMSS mentioned above and   $\{F^p\}_{p\geq 0}$ the filtration of the target $H^*_{S^1}(LX;\Q)$ associated with the EMSS.
Then, we have a decomposition
\[
  H^*_{S^1}(LX;\Q) = (H^*_{S^1}(LX;\Q)/F^1) \oplus F^1.
\]
Moreover, it follows that the map $i^* :  H^*_{S^1}(LX;\Q) \to H^*(LX;\Q)$ defined by the inclusion of the fibration
$LX \stackrel{i}{\to}  LX\times_{S^1} ES^1\to BS^1$ induces a monomorphism
\[
  i^* :  H^*_{S^1}(LX;\Q)/F^1 \to H^*(LX;\Q).
\]
We also recall the decomposition of the EMSS
\[
  \{E_r^{*,*}, d_r\} = \bigoplus_{N \in {\mathbb Z}}\{_{(N)}E_r^{*,*}, d_r\} \oplus \{\Q[u], 0\}
\]
introduced in \cite[Section 7]{KNWY}.

The derivation $\overline{L}_\theta$ is well-behaved in the vertical edge of the EMSS while it acts trivially apart from the edge.
As seen in the proof of the proposition below, the Cartan calculus, in particular, the contraction $e$ plays a crucial role in describing the property of  $\overline{L}$ in the EMSS.

\begin{prop}\label{prop:filtrations}
  For each $\theta$ in $\pi_*(\aut X)\otimes \Q$ and $p \geq 1$, $\overline{L}_\theta (F^p) \subset F^{p+1}$.
\end{prop}

In order to prove \cref{prop:filtrations},
we recall the Cartan calculus on a Sullivan algebra in \cref{sect:CartanOnSullivan}.
The proof of \cref{prop:CartanCalculusSullivanModel} allows us to obtain
\begin{lem} \label{lem:e-L} For each $\theta \in \pi_*(\aut X)\otimes \Q$, one has $[e_\theta, d] = 0$, $[e_\theta, B] = {}_aL_\theta$.
\end{lem}

\begin{proof}[Proof of \cref{prop:filtrations}] By \cref{lem:e-L}, we see that $[e_\theta u^{-1}, d + uB] = {}_aL_\theta$
  in $\Q^+[u]\cdot \cyclicModel$. This implies that $({}_a\overline{L}_{\theta})(F^p) = 0$ for
  $\theta$ in $\pi_*(\text{aut}_1(X))\otimes \Q$ and $p > 0$.
  By virtue of \cref{thm:LieAlg-New}, we have the result.
\end{proof}

\begin{prop}\label{prop:EMSS}
  \begin{enumerate}[wide]
    \item\label{item:EMSS-i}
      For $r \geq 2$ and $n > 1$, there exist morphisms of Lie algebras
      \begin{equation}\label{eq:EMSS}
        \overline{L}_{( \ )} : \pi_n(\aut X)\otimes \Q \to \Der_{n, 0}(E_r^{*, *})  \  \  \ \text{and}
      \end{equation}
      \begin{equation}\label{eq:EMSS_N}
        \overline{L}_{( \ )} : \pi_n(\aut X)\otimes \Q \to \End_{n, 0}({_{(N)}}E_r^{*, *})
      \end{equation}
      for which $\overline{L}_\theta$ is compatible with the differential $d_r$ and respects to the derivation $\overline{L}_\theta$ on $H^*_{S^1}(LX; \Q)$ in the
      $E_\infty$-term for each $\theta \in \pi_n(\aut X)\otimes \Q$. Moreover, up to isomorphism,
      the morphism $\overline{L}_{( \ )}$ of Lie algebras coincides with the map ${}_a\overline{L}_{( \ )}$.
    \item\label{item:EMSS-ii}
      The map $\overline{L}_{( \ )}$ acts trivially on $E_r^{*, q}$ for $q >1$.  As a consequence, for $\theta \in \pi_n(\aut X)\otimes \Q$, one has a commutative diagram
      \[
      \xymatrix@C20pt@R20pt{
      H^*_{S^1}(LX;\Q)/ F^1 \ar[r]^-{i^*} \ar[d]_{\overline{L_\theta}}& H^*(LX;\Q) \ar[d]^{L_\theta} \\
      H^{*-n}_{S^1}(LX;\Q)/ F^1 \ar[r]^-{i^*} & H^{*-n}(LX;\Q).
      }
      \]
  \end{enumerate}
\end{prop}

\begin{proof}
  We first observe that the multiplication on $E_r^{*,*}$ induces the map
  ${_{(N)}}E_r^{p, q} \otimes {_{(N')}}E_r^{p', q'} \to {_{(N+ N')}}E_r^{p+p', q+q'}$. The definition of the decomposition gives the result; see the discussion after \cite[Remark 7.1]{KNWY}.

  We use a rational model $\cyclicModel$ for the Borel construction $LX \times_{S^1} ES^1$ described above. Then, for an element $\theta \in \pi_n(\aut X)$,
  the morphism ${\mathcal M}(\overline{v_\theta}) : \cyclicModel \to H^*(S^n;\Q)\otimes \cyclicModel$ of CDGAs in the proof of \cref{thm:LieAlg-New} gives rise to
  a linear map $\overline{L}_\theta : E_r^{p, q} \to E_r^{p-n, q}$ for $p, q \geq 0$. In fact, the map ${\mathcal M}(\overline{v_\theta})$ preserves the filtration which constructs the EMSS. Therefore, we also see that $\overline{L}_\theta$ is compatible with the differential of each term of the EMSS.
  The equality \cref{eq:onE} enables us to deduce that the map $\overline{L}_\theta$ on $E_1^{*,*}$ coincides with the derivation
  ${}_aL_\theta$. The map ${}_aL_{( \ )} : \pi_{*}(\aut X)\otimes \Q \to \Der^{-*, 0}(E_r^{*',*'})$ is a morphism of Lie algebras and the so is
  $\overline{L}_{( \ )}$. This completes the proof of \cref{item:EMSS-i}.

  \cref{item:EMSS-ii} Let $\theta$ be a representative of an element in $\pi_n(\aut X)\otimes \Q$.
  A mentioned in the proof of \cref{prop:filtrations}, it follows that ${}_aL_\theta (x) = 0$ for $x \in E_1^{*, q}$ with $q >1$. In fact, such element $x$ is represented by one in the ideal $\Q^+[u]\cdot \cyclicModel$. Thus, the first half of the assertion of \cref{item:EMSS-ii} follows  from the result \cref{item:EMSS-i}.

  As for the latter half of the assertion, we have a commutative diagram
  \[
    \xymatrix@C20pt@R20pt{
      H^*_{S^1}(LX;\Q)/ F^1 \ar[d]_{\overline{L_\theta}} \ar[r]_-{\cong} \ar@/^12pt/[rr]^{i^*}& E_\infty^{*, 0} \ar[d]^{{}_aL_\theta}  \   \ar@{>->}[r] & H^*(LX;\Q) \ar[d]^{L_\theta} \\
      H^{*-n}_{S^1}(LX;\Q)/ F^1  \ar[r]^-\cong  \ar@/_12pt/[rr]_{i^*} & E_\infty^{*-n, 0}  \   \ar@{>->}[r] & H^{*-n}(LX;\Q).
    }
  \]
  The commutativity of the diagrams containing $i^*$ follows from a property of the EMSS.
  By the definition of ${}_aL_\theta$ and  \cref{prop:LieAlg-New}, we see that the right-hand side diagram is commutative.
  It follows from \cref{item:EMSS-i} that  the left-hand side diagram is commutative. We have the result.
\end{proof}

\begin{cor}\label{cor:cocycles}
  Let $x$ be an element in the image of the derivation $\overline{L}_\theta : E_r^{0, *} \to  E_r^{0, *}$ for some $\theta$. Then $d_r(x) = 0$.
\end{cor}

\begin{proof} The operation $\overline{L}_\theta$ is compatible with the differential $d_r$.
  \cref{prop:EMSS} \cref{item:EMSS-ii} implies the result.
\end{proof}

\bibliographystyle{alpha_plainsort}
\bibliography{bibliography}

\begin{thebibliography}{KNWY21}

\bibitem[AL00]{AL}
M.~{Arkowitz} and G.~{Lupton}.
\newblock {Rational obstruction theory and rational homotopy sets}.
\newblock {\em {Math. Z.}}, 235:525--539, 2000.

\bibitem[BL05]{B-L05}
J.~Block and Andrey Lazarev.
\newblock Andr{\'e}--{Quillen} cohomology and rational homotopy of function
  spaces.
\newblock {\em Adv. Math.}, 193(1):18--39, 2005.

\bibitem[BG76]{BG}
A.~K. {Bousfield} and V.~K. A.~M. {Gugenheim}.
\newblock {\em {On PL De Rham theory and rational homotopy type}}, volume 179.
\newblock Providence, RI: American Mathematical Society (AMS), 1976.

\bibitem[BS97]{BS97}
Edgar H.~jun. Brown and Robert~H. Szczarba.
\newblock {On the rational homotopy type of function spaces}.
\newblock {\em {Trans. Am. Math. Soc.}}, 349(12):4931--4951, 1997.

\bibitem[BM06]{BM06}
Urtzi {Buijs} and Aniceto {Murillo}.
\newblock {Basic constructions in rational homotopy theory of function spaces}.
\newblock {\em {Ann. Inst. Fourier}}, 56(3):815--838, 2006.

\bibitem[BM08]{BM08}
Urtzi {Buijs} and Aniceto {Murillo}.
\newblock {The rational homotopy Lie algebra of function spaces}.
\newblock {\em {Comment. Math. Helv.}}, 83(4):723--739, 2008.

\bibitem[BV88]{BV88}
Dan {Burghelea} and Micheline {Vigu\'e Poirrier}.
\newblock {Cyclic homology of commutative algebras. I}.
\newblock {Algebraic topology, rational homotopy, Proc. Conf.,
  Louvain-la-Neuve/Belg. 1986, Lect. Notes Math. 1318, 51-72 (1988).}, 1988.

\bibitem[CS99]{CS}
Moira {Chas} and Dennis {Sullivan}.
\newblock String topology.
\newblock preprint, arXiv:math/9911159v1, 1999.

\bibitem[{Con}85]{Connes}
Alain {Connes}.
\newblock {Non-commutative differential geometry}.
\newblock {\em {Publ. Math., Inst. Hautes \'Etud. Sci.}}, 62:41--144, 1985.

\bibitem[FHT01]{FHT1}
Yves {F\'elix}, Stephen {Halperin}, and Jean-Claude {Thomas}.
\newblock {\em {Rational homotopy theory}}, volume 205.
\newblock New York, NY: Springer, 2001.

\bibitem[FLS10]{FLS}
Yves {F\'elix}, Gregory {Lupton}, and Samuel~B. {Smith}.
\newblock {The rational homotopy type of the space of self-equivalences of a
  fibration}.
\newblock {\em {Homology Homotopy Appl.}}, 12(2):371--400, 2010.

\bibitem[FT04]{FT04}
Yves {F\'elix} and Jean-Claude {Thomas}.
\newblock {Monoid of self-equivalences and free loop spaces}.
\newblock {\em {Proc. Am. Math. Soc.}}, 132(1):305--312, 2004.

\bibitem[FK20]{FK}
Domenico {Fiorenza} and Niels {Kowalzig}.
\newblock {Higher brackets on cyclic and negative cyclic (co)homology}.
\newblock {\em {Int. Math. Res. Not.}}, 2020(23):9148--9209, 2020.

\bibitem[GJ90]{G-J}
Ezra {Getzler} and John D.~S. {Jones}.
\newblock {A\({}_{\infty}\)-algebras and the cyclic bar complex}.
\newblock {\em {Ill. J. Math.}}, 34(2):256--283, 1990.

\bibitem[{Hae}82]{Haefliger82}
Andre {Haefliger}.
\newblock {Rational homotopy of the space of sections of a nilpotent bundle}.
\newblock {\em {Trans. Am. Math. Soc.}}, 273:609--620, 1982.

\bibitem[Hal83]{Halperin83}
S.~Halperin.
\newblock Lectures on minimal models.
\newblock {\em M{\'e}m. Soc. Math. Fr., Nouv. S{\'e}r.}, 9-10:261, 1983.

\bibitem[HMR75]{HMR75}
Peter Hilton, Guido Mislin, and Joe Roitberg.
\newblock {\em Localization of nilpotent groups and spaces}, volume~15 of {\em
  North-Holland Math. Stud.}
\newblock Elsevier, Amsterdam, 1975.

\bibitem[HKO08]{HKO08}
Yoshihiro Hirato, Katsuhiko Kuribayashi, and Nobuyuki Oda.
\newblock A function space model approach to the rational evaluation subgroups.
\newblock {\em Math. Z.}, 258(3):521--555, 2008.

\bibitem[{Kur}06]{Kuribayashi06}
Katsuhiko {Kuribayashi}.
\newblock {A rational model for the evaluation map}.
\newblock {\em {Georgian Math. J.}}, 13(1):127--141, 2006.

\bibitem[KNWY21]{KNWY}
Katsuhiko {Kuribayashi}, Takahito {Naito}, Shun {Wakatsuki}, and Toshihiro
  {Yamaguchi}.
\newblock A reduction of the string bracket to the loop product.
\newblock {\em {Algebraic Geom. Topol.}} (2021), in press, arXiv:2109.10536v1.

\bibitem[KY97]{KY97}
Katsuhiko {Kuribayashi} and Toshihiro Yamaguchi.
\newblock The cohomology algebra of certain free loop spaces.
\newblock {\em {Fundamenta Mathematicae}}, 154:57--73, 1997.

\bibitem[KY00]{KY00}
Katsuhiko Kuribayashi and Toshihiro Yamaguchi.
\newblock On additive {K}-theory with the {L}oday--{Q}uillen *-product.
\newblock {\em Math. Scand.}, 87:5--21, 2000.

\bibitem[LS08]{LS08}
Pascal {Lambrechts} and Don {Stanley}.
\newblock {Poincar\'e duality and commutative differential graded algebras}.
\newblock {\em {Ann. Sci. \'Ec. Norm. Sup\'er. (4)}}, 41(4):497--511, 2008.

\bibitem[{Lod}98]{Loday}
Jean-Louis {Loday}.
\newblock {\em {Cyclic homology.}}, volume 301.
\newblock Berlin: Springer, 1998.

\bibitem[{Men}11]{Luc2011}
Luc {Menichi}.
\newblock {Van Den Bergh isomorphisms in string topology}.
\newblock {\em {J. Noncommut. Geom.}}, 5(1):69--105, 2011.

\bibitem[{Nai}24]{Nai24}
Takahito {Naito}.
\newblock {Cartan calculus in string topology}.
\newblock {\em {Proc. Am. Math. Soc.}}, 152(4):1789--1801, 2024.

\bibitem[NW22]{NW2022}
Takahito {Naito} and Shun {Wakatsuki}.
\newblock {`On loop-slant operations' in preparation}, 2022.

\bibitem[SYH18]{SYH}
Kazuhisa {Shimakawa}, Kohei {Yoshida}, and Tadayuki {Haraguchi}.
\newblock {Homology and cohomology via enriched bifunctors}.
\newblock {\em {Kyushu J. Math.}}, 72(2):239--252, 2018.

\bibitem[{Ste}67]{Steenrod}
N.~E. {Steenrod}.
\newblock {A convenient category of topological spaces}.
\newblock {\em {Mich. Math. J.}}, 14:133--152, 1967.

\bibitem[Sul77]{S}
Dennis Sullivan.
\newblock Infinitesimal computations in topology.
\newblock {\em Inst. Hautes \'Etudes Sci. Publ. Math.}, (47):269--331 (1978),
  1977.

\bibitem[{Vig}94]{Vigue1994}
Micheline {Vigu\'e-Poirrier}.
\newblock {Homologie cyclique des espaces formels}.
\newblock {\em {J. Pure Appl. Algebra}}, 91(1-3):347--354, 1994.

\bibitem[VPB85]{V-B1}
Micheline Vigu{\'e}-Poirrier and Dan Burghelea.
\newblock A model for cyclic homology and algebraic {K}-theory of 1-connected
  topological spaces.
\newblock {\em J. Differ. Geom.}, 22:243--253, 1985.

\bibitem[VPS76]{V-S}
Micheline Vigue-Poirrier and Dennis Sullivan.
\newblock The homology theory of the closed geodesic problem.
\newblock {\em J. Differ. Geom.}, 11:633--644, 1976.

\bibitem[{Wak}]{Wakatsuki2}
Shun {Wakatsuki}.
\newblock Kohomology.
\newblock \url{https://github.com/shwaka/kohomology}.

\bibitem[{Whi}78]{Whitehead}
George~W. {Whitehead}.
\newblock {\em {Elements of homotopy theory}}, volume~61.
\newblock Springer, New York, NY, 1978.

\bibitem[Yam05]{Yam05}
Toshihiro Yamaguchi.
\newblock {An example of fiber in fibrations whose Serre spectral sequences
  collapse}.
\newblock {\em {Czechoslovak Math. J.}}, 55(130):997--1001, 2005.

\end{thebibliography}

\end{document}